\documentclass[12pt,a4paper]{amsart}
\usepackage{amsmath}
\usepackage[english]{babel}

\usepackage[T1]{fontenc}
\usepackage{enumitem}

\usepackage{ifxetex}
\ifxetex \else
\usepackage[applemac]{inputenc}
\fi

\usepackage{amsmath}
\usepackage{amsxtra}
\usepackage{mathrsfs}
\usepackage{amssymb}
\usepackage[dvips]{graphicx}
\usepackage{graphicx}

\input xy
\xyoption{all}

\newtheorem{teo}{Theorem}[section]
\newtheorem{lemma}[teo]{Lemma}

\newtheorem{defi}[teo]{Definition}
\newtheorem{coro}[teo]{Corollary}

\newtheorem{cla}[teo]{Claim}
\newtheorem{prop}[teo]{Proposition}
\newtheorem{question}{Question}

\theoremstyle{remark}

\begin{document}

\newcommand{\ran}{{\rm ran}}
\newcommand{\cof}{{\rm cof}}
\newcommand{\dom}{{\rm dom}}
\newcommand{\I}{I}
\newcommand{\N}{\mathcal{N}}
\newcommand{\up}{\upharpoonright}
\newcommand{\rest}{\upharpoonright}
\newcommand{\urltilde}{\kern -.15em\lower .7ex\hbox{~}\kern .04em}

\newcommand\M{\mathcal{M}}
\newcommand\Acal{\mathscr{A}}
\newcommand\Bcal{\mathscr{B}}
\newcommand\Dcal{\mathscr{D}}
\newcommand\Ecal{\mathcal{E}}
\newcommand\Fcal{\mathscr{F}}
\newcommand\Hcal{\mathscr{H}}
\newcommand\Ical{\mathscr{I}}
\newcommand\Ncal{\mathscr{N}}
\newcommand\Mcal{\mathscr{M}}
\newcommand\Pcal{\mathscr{P}}
\newcommand\Qcal{\mathscr{Q}}
\newcommand\calQ{\mathcal{Q}}
\newcommand\Rcal{\mathscr{R}}
\newcommand\Scal{\mathcal{S}}
\newcommand\Tcal{\mathcal{T}}
\newcommand\Ucal{\mathcal{U}}
\newcommand\Vcal{\mathscr{V}}
\newcommand\Zcal{\mathscr{Z}}
\newcommand\Rbb{\mathbb{R}}
\newcommand\Nfrak{\mathfrak{N}}
\newcommand\Pfrak{\mathfrak{P}}
\newcommand\restrict{\restriction}
\newcommand{\seq}{\subseteq}
\newcommand{\ORD}{{\rm ORD}}

\newcommand\Gdot{\dot{G}}
\newcommand\Qdot{\dot{Q}}

\newcommand{\diff}{\operatorname{\mathrm diff}}
\newcommand{\Ht}{\operatorname{\mathrm ht}}
\newcommand{\lev}{\operatorname{\mathrm lev}}
\newcommand{\meet}{\wedge}
\newcommand\triord{\triangleleft}
\newcommand{\Th}{{}^{\mathrm{th}}}
\newcommand{\St}{{}^{\mathrm{st}}}
\newcommand\axiom{\mathrm}
\newcommand\MA{\axiom{MA}}
\newcommand\MM{\axiom{MM}}
\newcommand\PFA{\axiom{PFA}}
\newcommand\BPFA{\axiom{BPFA}}
\newcommand\MRP{\axiom{MRP}}
\newcommand\SRP{\axiom{SRP}}
\newcommand\ZFC{\axiom{ZFC}}
\newcommand\CAT{\axiom{CAT}}
\newcommand\CH{\axiom{CH}}
\newcommand{\<}{\langle}
\renewcommand{\>}{\rangle}
\newcommand\mand{\textrm{ and }}
\renewcommand{\diamond}{\diamondsuit}
\newcommand{\forces}{\Vdash}

\newcommand{\cf}{{\mbox{cof}}}
\newcommand{\height}{\Ht}
\newcommand\NS{\mathrm{NS}}

\def\0{\mathbf 0}
\def\1{\mathbf 1}

\newcommand{\J}{{\mathcal J}}
\newcommand{\FF}{{\mathcal F}}
\newcommand{\B}{{\mathcal B}}
\newcommand{\E}{{\mathcal E}}
\newcommand{\CC}{{\mathcal C}}
\newcommand{\DD}{{\mathcal D}}
\newcommand{\A}{{\mathcal A}}
\newcommand{\re}{\restriction}
\newcommand{\wh}{\; \widehat \;}

\title{Ranks of Maharam algebras}
\author{\v{Z}ikica Perovi\'{c}}
\address{ MiraCosta College, One Barnard Drive, Oceanside, CA 92056, USA }
\email{zperovic@miracosta.edu}

\author[Boban Veli\v{c}kovi\'{c}]{Boban Veli\v{c}kovi\'{c}}
\address{Institut de Math\'ematiques de Jussieu - Paris Rive Gauche,
Universit\'e Paris Diderot,  8 Place Aur\'elie Nemours, 75205 Paris Cedex 13, France}
\email{boban@math.univ-paris-diderot.fr}
\urladdr{http://www.logique.jussieu.fr/~ boban}


\begin{abstract} Solving a well-known problem of Maharam, Talagrand \cite{Talagrand}  
constructed an exhaustive  non uniformly exhaustive submeasure, thus also providing the first example of 
a Maharam algebra that is not a measure algebra. 
To each exhaustive submeasure one can canonically assign a certain countable ordinal, its exhaustivity rank.
In this paper, we use carefully constructed Schreier families and norms
derived from them to provide examples of exhaustive submeasures 
of arbitrary high exhaustivity rank. This gives rise to uncountably many
non isomorphic separable atomless Maharam algebras. 
\end{abstract}

\maketitle

\section{ Introduction}
We say that a complete Boolean ${\mathcal B}$ algebra is a {\em measure
algebra} if it admits a strictly positive $\sigma$-additive
probability measure.
Recall that a {\em
submeasure} on Boolean algebra ${\mathcal B}$ is a function $\nu
:{\mathcal B}\rightarrow [0,+\infty]$ such that

\begin{enumerate}

\item $\nu({\0})=0$,
\item If $x\leq y$ then $\nu(x)\leq \nu(y)$,
\item$\nu(x\vee y)\leq \nu(x)+\nu(y)$, for all $x,y \in \mathcal B$;
\end{enumerate}

\noindent We say that $\nu$ is {\em positive} if $\nu(a)>0$, for
every $a\in {\mathcal B}\setminus \{ {\0}\}$. If ${\mathcal B}$ is
complete the role of $\sigma$-additivity is played by the following
continuity condition.

\begin{enumerate}
\item[(4)] $\nu(x_n)\rightarrow \nu(\inf_nx_n)$, whenever $\{x_n\}_n$ is a decreasing
sequence.
\end{enumerate}

\noindent A submeasure $\nu$ satisfying (4) is called {\em
continuous}. If a complete Boolean algebra ${\mathcal B}$ carries a
positive continuous submeasure then we call it a {\em Maharam
algebra}.

In an attempt to find an algebraic characterization of measure algebras  Von Neumann asked in 1937 
if every ccc weakly distributive  complete  Boolean algebra is a measure algebra 
(see \cite{Mau:ScottishBook}).   Working on Von Neumann's
problem Maharam \cite{Maharam} formulated the notion of a
continuous submeasure and found an algebraic characterization for a
complete Boolean algebra to carry one. Maharam also showed that every Maharam
algebra is weakly distributive and satisfies the ccc. Therefore Von
Neumann's original question was naturally decomposed into two questions.

\begin{question} Is every Maharam algebra a measure algebra?
\end{question}

\begin{question} Is every  ccc weakly distributive complete Boolean
algebra a Maharam algebra?
\end{question}
In this paper we will not discuss Question 2, instead we refer the interested reader to \cite{Ve:survey}. 
Over the years a significant amount of work has been done on  Question 1, which was known to be equivalent 
to the famous Control Measure Problem, i.e. the question whether every countably additive
vector valued measure $\mu$ defined on a $\sigma$-algebra of sets
and taking values in an $F$-space, i.e. a completely metrizable
topological vector space, admits a {\em control measure}, i.e. a
countable additive scalar measure $\lambda$ having the same null
sets as $\mu$. For instance,   Kalton and Roberts
\cite {Kalton-Roberts} showed that a submeasure $\mu$ defined on a
(not necessarily complete) Boolean algebra ${\mathcal B}$ is equivalent
to a measure if and only if it is uniformly exhaustive. Recall that
a submeasure $\mu$ on  a Boolean algebra ${\mathcal B}$ is called {\em
exhaustive} if for every sequence $\{a_n\}_n$ of disjoint elements
of ${\mathcal B}$ we have $\lim _n \mu(a_n)=0$, $\mu$ is called
 {\em uniformly exhaustive} if for every $\epsilon >0$ there is an integer $n$ such
that there is no sequence of $n$ pairwise disjoint elements of ${\mathcal B}$ of
$\mu$-submeasure $\geq \epsilon$. Clearly, every continuous
submeasure on a complete Boolean algebra is exhaustive. If $\mu$ is
a positive submeasure on a Boolean algebra ${\mathcal B}$ one can
define a metric $d$ on ${\mathcal B}$ by setting $d(a,b)=\mu(a\Delta
b)$. If $\mu$ is exhaustive then the metric completion
$\bar{\mathcal B}$ of ${\mathcal B}$ equipped with the natural
boolean algebraic structure is a complete Boolean algebra and $\mu$
has a unique extension $\bar{\mu}$ to a continuous submeasure on
$\bar{\mathcal B}$. Thus $\bar{\mathcal B}$ is a Maharam algebra. It
follows that  Question 1  is equivalent to the question whether every
exhaustive submeasure on a Boolean algebra ${\mathcal B}$ is uniformly
exhaustive. In 2005  Talagrand \cite{Talagrand} produced a remarkable
example of an exhaustive submeasure which is not uniformly
exhaustive. As a consequence he obtained the following result.

\begin{teo}[\cite{Talagrand}]
 There is a Maharam algebra which is not a measure algebra.
\end{teo}
Now we know that there are Maharam algebras that are not measure algebras, 
but we do not know much about their structure. Fremlin (see \cite{measure-theory-v5}) 
suggested using the exhaustivity rank as a tool for classifying Maharam algebras.

Suppose that $\B$ is a Boolean algebra and $\nu$ an exhaustive submeasure on $\B$. For $\epsilon > 0$, 
let $\mathcal{D}_\epsilon (\nu)$ be the set of all finite pairwise disjoint subsets $F$ of
$\B$ such that $\nu(a)\geq \epsilon$, for all $a\in F$.
Since $\nu$ is exhaustive it follows that $(\mathcal{D}_\epsilon (\nu),\supset)$ is well-founded. 
Let ${\rm rk }_\epsilon(\nu)$ be the rank of this ordering. More precisely, 
for each $F\in  \mathcal D_\epsilon(\nu)$, we define the ${\rm rk}_\epsilon(\nu,F)$ 
 by letting: 
\[
{\rm rk}_\epsilon(\nu,F)=\sup\{{\rm rk}_\epsilon(\nu,G)+1:G \in \mathcal D_\epsilon(\nu)
\mbox{ and } G \supsetneq F\}
\]
We then let ${\rm rk}_\epsilon(\nu)={\rm rk}_\epsilon(\nu,\emptyset)$. Finally, we let
${\rm rk}(\nu)= \sup \{ {\rm rk}_\epsilon(\nu): \epsilon >0\}$. 
Since any two Maharam submeasures on a Maharam algebra $\B$ are absolutely continuous with
respect to each other they have the same exhaustivity rank, hence this rank is an invariant
of $\B$ and we denote it by  ${\rm rk}(\B)$. 
Fremlin \cite{Fremlin-06204}	
proved that if $\B$ is a Maharam algebra, but not a measure algebra then  ${\rm rk}(\B)\ge \omega^\omega$. 
He also showed that ${\rm rk}(\mathcal T) \leq\omega^{\omega^2}$ for
the Maharam algebra $\mathcal T$ constructed by Talagrand \cite{Talagrand}.
Generalizing  Fremlin's question (\cite {measure-theory-v5} 539Z) we consider  the following.
\begin {question} Are there Maharam algebras of arbitrary high countable exhaustivity rank? 
\end{question}
We give a positive answer to this question. We define the notion of an admissible norm
and we generalize Talagrand's construction by replacing the cardinalities of the relevant sets by their norms. 
By varying this norm we obtain  examples of submeasures of arbitrary high exhaustivity ranks.

The paper is organized as follows. In \S 2 we define admissible norms and show how to produce
examples of such norms using Schreier families.  We also prove some easy technical facts
about these norms that will be needed in the main construction. In \S 3 we describe our
generalization of Talagrand's construction based on any admissible norm. We also give a lower
bound on the exhaustivity ranks of the submeasures built on the Schreier norms. 
In \S 4 we prove that the submeasures constructed in \S 3 are exhaustive and derive some corollaries.
In \S 5 we provide upper bounds on the exhaustivity ranks of our submeasures. 
Our presentation is completely self-contained, however a good understanding of  \cite{Fa:Examples},
\cite{Ro:Maharam}, and \cite{Talagrand} would clearly be useful when reading the current paper.

\section{Admissible families and norms}

We will be interested in functions on finite sets of integers that have certain features of the cardinality function. 
 
\begin{defi}\label{spreading}
Suppose $A$ and $B$ are finite subsets of $\mathbb N$. We write $A\leq_s B$ if,
 letting $A=\{ a_0,\ldots, a_{n-1}\}$ and $B=\{ b_0,\ldots b_{m-1} \}$
 be the increasing enumerations of $A$ and $B$, we have that $n=m$ and $a_i\leq b_i$, for all 
 $i<n$. 
\end{defi}

\begin{defi}\label{admissible-norms} A {\em norm} is a function $\Vert \cdot \Vert : [\mathbb N]^{<\omega} \rightarrow
	\mathbb N$ such that:
	\begin{enumerate}
		\item[$(1)$] $|| \emptyset || = 0$ and $|| \{ n\}|| =1$, for all $n\in \mathbb N$,
		\item[$(2)$] if $A\subseteq B$ then $|| A|| \leq ||B||$,
		\item[$(3)$] $|| A\cup B|| \leq || A || + || B||$, for every $A,B\in [\mathbb N]^{<\omega}$.
	\end{enumerate}
	 We say that a norm $||\cdot ||$ is: 
	 \begin{enumerate}
	 	\item[$(4)$] {\em unbounded} if $\lim_{n\rightarrow \infty} || A\cap n|| = +\infty$,
	 	for every infinite $A\subseteq \mathbb N$,
	 	\item[$(5)$] {\em spreading} if $|| A||\leq ||B||$, for every 
	 	 $A,B \in [\mathbb N]^{<\omega}$ such that $A\leq_s B$. 
	 
	 \end{enumerate}
	 A norm that is both unbounded and spreading will be called {\em admissible}.
\end{defi}

We now describe a canonical way to generate admissible norms on $[\mathbb N]^{<\omega}$.  

\begin{defi}\label{admissible-families} Let $\mathscr S$ be a family of finite subsets of $\mathbb N$. 
We say that $\mathscr S$ is:
	\begin{enumerate}
	\item[$(1)$] {\em hereditary} if it is closed under taking subsets. 
	\item[$(2)$] {\em spreading} if  $A\in \mathscr S$ and $A \leq_s B$ implies $B\in \mathscr S$.
	\item[$(3)$]  {\em compact} if it is a compact subset of $2^{\mathbb N}$ with the product topology, 
	where we identify a subset of $\mathbb N$ with its characteristic function. 
	\end{enumerate}
      Finally,  we say that $\mathscr S$ is {\em admissible} if it is compact, hereditary, spreading and contains all singletons.
\end{defi}

Suppose $\mathscr S$ is an admissible family of finite subsets of $\mathbb N$.  
We can define a norm $|| \cdot ||_{\mathscr S}$ by letting $||A||_{\mathscr S}$ be the least
number of members of $\mathscr S$ needed to cover $A$. It is straightforward to check
that $|| \cdot ||_{\mathscr S}$ is an admissible norm. Conversely, if $|| \cdot ||$
is an admissible norm we can let $\mathscr S = \{ F\in [\mathbb N]^{<\omega}: || F||\leq 1\}$.
Then $\mathscr S$ is an admissible family and $|| \cdot || = || \cdot ||_{\mathscr S}$.
We can assign a rank to each admissible family $\mathscr S$. We do this using the language
of games. 

\begin{defi}\label{family-game} Let $\mathscr S$ be an admissible family and $\alpha$
	a countable ordinal. The game $\mathcal G_\alpha(\mathscr S)$ is played between two players
	${\rm I}$ and ${\rm II}$ as follows.
\[
\begin{array}[c]{ccccccccc}
{\rm I} : & \alpha_0 \phantom{n_0} & \alpha_1 \phantom{n_1}& \ \ \cdots \ \ & \alpha_k \phantom{n_k}& \ \ \cdots \ \ \\
\hline
{\rm II} : & \phantom{\alpha_0} n_0 & \phantom{\alpha_1} n_1 & \ \ \cdots \ \ & \phantom{\alpha_k} n_k  & \ \ \cdots \ \ \\
\end{array}
\]

\noindent Player ${\rm I}$ is required to play a decreasing sequence of ordinals $\leq \alpha$
and Player ${\rm II}$ is required to play an increasing sequence of integers such
that $\{n_0,\ldots n_k \} \in \mathscr S$, for all $k$.
The first player who cannot play loses. 
\end{defi}

Since Player ${\rm I}$ plays a decreasing sequence of ordinals, the game must end after finitely many stages. 
Therefore, by the Gale-Stewart theorem \cite{GaleStewart} one of the players has
a winning strategy. Since every infinite subset of $\mathbb N$ has an initial segment
which is not in $\mathscr S$, Player ${\rm II}$ cannot have a winning strategy in $\mathcal G_\alpha(\mathscr S)$, 
for all $\alpha< \omega_1$. We let $\rho(\mathscr S)$ be the least
$\alpha$ such that Player ${\rm I}$ has a winning strategy in  $\mathcal G_\alpha(\mathscr S)$.
Let $T_{\mathscr S}$ be the set of strictly increasing sequences of integers
whose range is in $\mathscr S$. We order $T_\mathscr S$ by reverse extension,
i.e. $s < t$ iff $t$ is a proper initial segment of $s$. Then $T_{\mathscr S}$
is well-founded and $\rho (\mathscr S)$ is simply the well-founded rank of $T_{\mathscr S}$. 

Given families  $\mathscr S$ and $\mathscr T$ of subsets of $\mathbb N$ 
let 
$
\mathscr S \oplus \mathscr T= \{ S \cup T: S \in \mathscr S \mbox{ and } T \in \mathscr T \}.
$  
It is easy to see that, if $\mathscr S$ and $\mathscr T$ are admissible, then so is 
$\mathscr S \oplus \mathscr T$.

\begin{lemma}\label{oplus} Suppose $\mathscr S$ and $\mathscr T$ are admissible families.
Let $\rho (\mathscr S)=\alpha$ and $\rho(\mathscr T)=\beta$. Then 
$\rho (\mathscr S \oplus \mathscr T)\leq (\alpha +1)(\beta +1)-1$.
\end{lemma}

\begin{proof} Let us fix winning strategies $\sigma$ and $\tau$ for Player {\rm I} in $\mathcal G_{\alpha}(\mathscr S)$
and $\mathcal G_{\beta}(\mathscr T)$. Let $\gamma = (\alpha +1)(\beta +1)-1$.
We need to define a winning strategy  for Player {\rm I}
in $\mathcal G_{\gamma}(\mathscr S \oplus \mathscr T)$. 
Note that the lexicographic ordering $<_{\rm lex}$ on $(\beta +1)\times (\alpha +1)$
has order type $(\alpha +1)(\beta +1)$. So, instead of playing ordinals $\leq \gamma$
Player {\rm I} will play pairs of ordinals in $(\beta +1)\times (\alpha +1)$ decreasing under $<_{\rm lex}$.
Player {\rm I} starts by playing according to $\sigma$, but at any give stage 
instead of playing the ordinal $\xi$ given by $\sigma$ he plays $(\beta,\xi)$. 
Player {\rm II} plays an increasing sequence of integers $\{n_0, n_1, \ldots\}$. 
Since $\sigma$ is a winning strategy for Player {\rm I} in $\mathcal G_{\alpha}(\mathscr S)$
there must be a stage $k_0$ such that $\{ n_0,\ldots, n_{k_0}\} \notin \mathscr S$. 
At that moment Player {\rm I} switches to playing $\mathcal G_{\beta}(\mathscr T)$
and considers that Player {\rm II} has played $n_{k_0}$ as the first move in this game. 
Suppose $\tau$ replies by playing some $\beta_1 < \beta$. 
Player {\rm I} then starts a new run of $\mathcal G_{\alpha}(\mathscr S)$ in which he plays
pairs of the form $(\beta_1,\xi)$, for $\xi \leq \alpha$. Since $\sigma$ is a winning strategy in this game, 
there must be a first stage $k_1$ such that $\{ n_{k_0+1},\ldots, n_{k_1}\} \notin \mathscr S$. 
Player {\rm I} then considers that Player {\rm II} has made another move 
 in $\mathcal G_{\beta}(\mathscr T)$ by playing $n_{k_1}$. Let $\beta_2 < \beta_1$
be the response of $\tau$. Player {\rm I} then starts yet another run of $\mathcal G_{\alpha}(\mathscr S)$
in which he plays pairs of the form $(\beta_2,\xi)$, for $\xi \leq \alpha$. 
Continuing in this way, we obtain increasing blocks of integers $B_0, B_1,\ldots$. 
Each block $B_i$ is of the form $\{ n_{k_{i-1}+1}, \ldots,n_{k_i}\}$. Here, we set by convention
$k_{-1}=-1$.  We have that $B_i\notin \mathscr S$, for each $i$. 
Since $\tau$ is a winning strategy for Player {\rm I} in $\mathcal G_{\beta}(\mathscr T)$,
by the time Player {\rm I} reaches  $(0,0)$, Player {\rm II} has played $l$ blocks $B_0,\ldots, B_{l-1}$
such that  $R=\{ n_{k_0}, \ldots, n_{k_{l-1}}\} \notin \mathscr T$. 
We claim that $\bigcup_{i<l}B_i \notin  \mathscr S \oplus \mathscr T$.
Indeed, suppose it could be written as  $S\cup T$, for some $S\in \mathscr S$ and $T \in \mathscr T$.  
Since $B_i\notin \mathscr S$,  there must be an element $m_i\in B_i\setminus S$, for all $i <l$.
Let $P=\{ m_0,\ldots, m_{l-1}\}$. Then we must have that $P\subseteq T$. Since $\mathscr T$ is hereditary, 
we would have  that $P\in \mathscr T$, as well. Now, note that $P\leq_s R$. Since $\mathscr T$ is
also spreading, we would get that $R\in \mathscr T$, a contradiction. 

\end{proof}

\begin{coro}\label{rank-sums} Suppose $\mathscr S$ is an admissible family and let
	 $\alpha=\rho(\mathscr S)$.
Let $\| \cdot \|_{\mathscr S}$ be the associated norm. Suppose $n>0$ is an integer and let
$\mathscr S^n= \{ F\in [\mathbb N]^{<\omega} : \| F\|_{\mathscr S} \leq n\}$. 
Then $\rho(\mathscr S^n)\leq (\alpha +1)^n-1$.
\qed
\end{coro}

We now define a version of the Schreier families initially introduced in \cite{Schreier}. 
These families have played an important role in the theory of Banach spaces, see for instance 
\cite{AlAr} or \cite{GasLeu}. For applications of Schreier families in combinatorics, 
see, for instance, \cite{Farmaki}.
Since we need our families to be spreading we have to take some care in their definition.
It will be convenient to use the following lemma of Galvin, see \cite{Hoshe} or \cite{Roitman}
for a proof.

\begin{lemma}
	\label{Galvin}  There is a sequence $(<_n)_n$ of tree orderings on $\omega_1$
	such that: 
	\begin{enumerate}
		\item[$(1)$] if $n<m$ and $\xi <_n \eta$ then $\xi <_m \eta$,
		\item[$(2)$] $(\omega_1,<_n)$ has finite height, for all $n$,
		\item[$(3)$] $<\restriction \omega_1 = \bigcup_n <_n$. 
		\qed
	\end{enumerate}  
	
\end{lemma}

We fix a Galvin decomposition $(<_n)_n$ such that $0<_n \xi$, for all $0< \xi < \omega_1$ and all $n$. 

\begin{defi} [Schreier families] Suppose $0 < \alpha < \omega_1$. 
We define the family $\mathscr S_\alpha$
as the collection of all  $F \in [\mathbb N]^{< \omega}$ such that, if 
$F=\{ n_0,\ldots, n_{k-1}\}$
is the increasing enumeration, there is a sequence $(\alpha_i)_{i\leq k}$ of 
ordinals such that $\alpha_0=\alpha$ and $\alpha_{i+1} <_{n_i} \alpha_i$, for all $i <k$.
\end{defi}

\begin{lemma}\label{Schreier} The family $\mathscr S_\alpha$ is admissible, for all countable ordinals $\alpha >0$. 
\end{lemma}

\begin{proof}
Let us first observe that if $F\in \mathscr S_\alpha$ then there is a canonical sequence
of ordinals witnessing it. Namely, suppose $F=\{ n_0,\ldots, n_{k-1}\}$ is the increasing
enumeration. We define the sequence $(\alpha_i)_{i\leq k}$ by induction as follows.
Let $\alpha_0=\alpha$. Suppose $\alpha_i$ has been defined. Since $(\omega_1,<_{n_i})$
is a tree of finite height, the set of $<_{n_i}$-predecessors of $\alpha_i$
is finite and totally ordered. If it is non-empty, we let $\alpha_{i+1}$ be the largest $<_{n_i}$-predecessor 
of $\alpha_i$. It is straightforward to check, by using (1) of Definition \ref{Galvin} and the fact that 
$F\in \mathscr S_\alpha$, that we can continue the construction up to
$k$. Notice that, also by (1) of  Definition \ref{Galvin}, the family $\mathscr S_\alpha$
is spreading and hereditary, for all $\alpha$. To see that $\mathscr S_\alpha$ is compact,
suppose $A$ is an infinite subset of $\mathbb N$ such that $A\cap n \in \mathscr S_\alpha$, for all $n$. 
Let $\{ n_0,n_1,\ldots\}$ be the increasing enumeration of $A$. Then, as before, we could construct
a sequence $(\alpha_k)_k$ of ordinals such that $\alpha_0=\alpha$ and $\alpha_{k+1}<_{n_k} \alpha_k$,
for all $k$. Then $(\alpha_k)_k$ would be an infinite decreasing sequence of ordinals, a contradiction. 
Finally, if $\alpha >0$, since we assumed that $\alpha >_n0$, for all $n$, it follows that $\mathscr S_\alpha$ contains all singletons. Therefore, $\mathscr S_\alpha$ is an admissible family.
Let us also note that if $\alpha <_n \beta$, $n < F$ and $F\in \mathscr S_\alpha$
then $F \cup \{ n\}\in \mathscr S_\beta$.
\end{proof}

\begin{lemma}\label{Schreier-ranks}
	\label{rho} $\rho(\mathscr S_\alpha)= \alpha$, for all countable ordinals $\alpha>0$.
\end{lemma}
\begin{proof}
Suppose $\alpha^* <\alpha$. We describe a winning strategy for Player {\rm II} in 
${\mathcal G}_{\alpha ^*}(\mathscr S_\alpha)$. We may assume that Player {\rm I} starts by playing 
 $\alpha_0 =\alpha^*$. Player {\rm II} lets $n_0$ be the least integer such that $\alpha_0 <_{n_0}\alpha$. 
At stage $k$, suppose Player {\rm I} plays some $\alpha_k <\alpha_{k-1}$.
Then Player {\rm II} lets $n_k$ be the least integer bigger than $n_{k-1}$
such that $\alpha_{k} <_{n_k} \alpha_{k-1}$. 
Since $\alpha >_{n_0}\alpha_0 > \ldots >_{n_k}\alpha_k$, it follows
that $\{ n_0,\ldots,n_k\} \in \mathscr S_\alpha$. This means that Player {\rm II}
can keep playing as long as Player {\rm I} keeps producing a decreasing sequence
of ordinals. Therefore, Player {\rm II} wins by playing in this way. 
	
Now we describe the winning strategy for Player {\rm I} in ${\mathcal G}_{\alpha}(\mathscr S_\alpha)$. 
He starts by playing $\alpha_0=\alpha$. Suppose Player {\rm II} responds
by playing some $n_0$. Since $\{ n_0\}\in \mathscr S_\alpha$ then $\alpha_0$ is
not a minimal element in $<_{n_0}$.  Let $\alpha_1$ be the largest $<_{n_0}$-predecessor
of $\alpha_0$. Player {\rm I} then plays $\alpha_1$. Suppose we are at some stage
$k$ and Player {\rm II} has played $n_{k-1}$ in the previous stage. 
Since $\{ n_0,\ldots, n_{k-1}\} \in \mathscr S_\alpha$ it follows that $\alpha_{k-1}$
is not a minimal element in $<_{n_{k-1}}$. Then Player {\rm I} plays as $\alpha_k$ 
the largest $<_{n_{k-1}}$-predecessor of $\alpha_{k-1}$. 
Since  $(\alpha_k)_k$ is a decreasing sequence of ordinals, the game must stop at some stage,
i.e. at some $k$ Player {\rm II} cannot find $n_k > n_{k-1}$ such that $\{n_0,\ldots, n_k\} \in \mathscr S_\alpha$. 
Therefore, Player {\rm I} wins by following this strategy. 
	
\end{proof}		

\begin{defi}\label{alpha-norm} We shall write $\| \cdot \|_\alpha$ for the norm derived
from the family $\mathscr S_\alpha$, for $\alpha <\omega_1$.
\end{defi}

We now turn to a different game that will be used to analyze the exhaustivity ranks
of our submeasures. 
	
\begin{defi}\label{disjoint-game} Suppose $\mathcal P$ is a poset and $\mathcal F \subseteq \mathcal P$. 
For an ordinal $\alpha$ the game $\mathcal H_{\alpha}(\mathcal F)$ is played between  players ${\rm I}$ 
and ${\rm II}$ as follows.
\[
\begin{array}[c]{ccccccccc}
{\rm I} : & \alpha_0 \phantom{n_0} & \alpha_1 \phantom{n_1}& \ \ \cdots \ \ & \alpha_k \phantom{n_k}& \ \ \cdots \ \ \\
\hline
{\rm II} : & \phantom{\alpha_0} p_0 & \phantom{\alpha_1} p_1 & \ \ \cdots \ \ & \phantom{\alpha_k} p_k  & \ \ \cdots \ \ \\
\end{array}
\]
Player {\rm I}  is required to plays decreasing sequence of ordinals $\leq \alpha$, while Player {\rm II}
plays pairwise incompatible members of $\mathcal F$. The first player who cannot play  loses.
\end{defi}
Clearly, if there is an infinite pairwise incompatible sequence of elements of $\mathcal F$ then Player {\rm II} 
has a winning strategy in $\mathcal H_{\alpha}(\mathcal F)$, for any $\alpha$. 
He simply plays the members of that sequence regardless of what Player {\rm I} plays.
If there is no such sequence of members of $\mathcal F$  then there is 
an ordinal $\alpha$ such that Player {\rm I} has a winning strategy. Let $\delta(\mathcal F)$
be the least such $\alpha$. In other words, if $\mathcal D(\mathcal F)$ is the family
of  pairwise incompatible finite subsets of $\mathcal F$ then $\delta(\mathcal F)$
is equal to ${\rm rk}(\mathcal D(\mathcal F))$, i.e. the rank of $\mathcal D(\mathcal F)$ under reverse inclusion.

In the next lemma and in \S 5 we will use the natural sum of ordinals, see \cite{Sierpinski}.
Recall that for ordinals $\alpha$ and $\beta$, the {\em natural sum} of $\alpha$ and $\beta$,
denoted by $\alpha \oplus \beta$ is defined by simultaneous induction on $\alpha$ and $\beta$
as the smallest ordinal greater than $\alpha \oplus \gamma$, for all $\gamma < \beta$, 
and $\gamma \oplus \beta$, for all $\gamma <\alpha$. 
Another way to define the natural sum  of two ordinals $\alpha$ and $\beta$
is to use the Cantor normal form: one can find a sequence of ordinals $\gamma_0 > \ldots \gamma_{n-1}$ 
and two sequences ($k_0,\ldots, k_{n-1})$ and $(j_0,\ldots, j_{n-1})$
of natural numbers (including zero, but satisfying $k_i + j_i > 0$, for all $i$) such that
$\alpha = \omega^{\gamma_0}\cdot k_0 + \cdots +\omega^{\gamma_{n-1}}\cdot k_{n-1}$
and $\beta = \omega^{\gamma_0}\cdot j_0 + \cdots +\omega^{\gamma_{n-1}}\cdot j_{n-1}$
and defines

\[
\alpha \oplus \beta = \omega^{\gamma_0}\cdot (k_0+j_0) + \cdots +\omega^{\gamma_{n-1}}\cdot 
(k_{n-1}+j_{n-1}). 
\]

What is important for us is that the natural sum is associative and commutative. It is always greater or equal 
to the usual sum, but it may be strictly greater.

\begin{defi}
Let $\mathcal P$ be the poset  of all partial functions $u$ such that $\dom(u)\in [\mathbb N]^{<\omega}$ and $u(k) < 2^k$, 
for all $k \in \dom(u)$, ordered under reverse inclusion. 
For $0<\alpha<\omega_1$, let  $\mathcal P_\alpha$ be the set of all $u \in \mathcal P$
with  $\|  \dom (u) \|_\alpha  \leq 1$. 
\end{defi}

\begin{lemma}\label{Rank} Suppose  $0< \alpha< \omega_1$. Then $\delta(\mathcal P_\alpha) =\omega^\alpha$.
\end{lemma}
\begin{proof}  We will give a proof by induction. First note that, since $\mathscr S_\beta$
is spreading, for all $\beta$, if  Player {\rm II} has a winning strategy in $\mathcal H_\gamma(\mathcal P_\beta)$ 
for some   $\gamma$, then for every integer $n$, Player {\rm II} has a winning strategy in the same game in which 
he plays  partial functions $u\in \mathcal P_\beta$ with $\min (\dom (u)) >n$. Now, suppose $\alpha$ is a countable ordinal
and the statement is true for all $\beta <\alpha$. 
	
Suppose $\xi < \omega^\alpha$. We describe informally a winning strategy for Player {\rm II}
in $\mathcal H_{\xi}(\mathcal P_\alpha)$. 
We may assume that Player {\rm I}'s starts by playing $\xi_0 =\xi$. 
We first find $\beta <\alpha$ and an integer $n_0$ such that 
$\xi_0 = \omega^\beta\cdot n_0 + \eta_0$,
for some $\eta_0 < \omega^\beta$. We can then find an integer $m$ such that $2^m > n_0$ and
$\beta <_m\alpha$. Note that if $F\in \mathscr S_\beta$
and $m < \min(F)$ then $\{m \} \cup F \in \mathscr S_\alpha$. Fix a winning strategy
$\tau_0$ for Player {\rm II} in $\mathcal H_{\eta_0}(\mathcal P_\beta)$ in which
he plays only partial functions $u\in \mathcal P_\beta$ with $\min(\dom(u)) >m$. Now, let $u_0$
be the response of $\tau_0$ if Player {\rm I} plays $\eta_0$ in the game 
$\mathcal H_{\eta_0}(\mathcal P_\beta)$. Player {\rm II} then plays 
$v_0=\{ (m,n_0)\} \cup u_0$. Note that if $F_0=\dom (u_0)$ then $F_0\in \mathscr S_\beta$
and hence $\{m\} \cup F_0\in \mathscr S_\alpha$. In other words $v_0\in \mathcal F_\alpha$. 
As long as  Player {\rm I} plays ordinals $\xi_i$ of the form $\omega^\beta\cdot n_0 + \eta_i$,
for some $\eta_i$, Player {\rm II} simulates the run of the game $\mathcal H_{\eta_0}(\mathcal P_\beta)$ in
which Player {\rm I} plays the $\eta_i$. At stage $i$, if $u_i$ is the response of $\tau_0$ in 
that game he plays $v_i=\{ (m,n_0)\} \cup u_i$ in the current game. 
Suppose that at some stage $i$ Player {\rm I} plays an ordinal $\xi_i$ of the form 
$\omega^\beta \cdot n_1 + \eta_i$ for some $n_1<n_0$ and $\eta_i <\omega^\beta$. Fix a winning
strategy $\tau_1$ for Player II in $\mathcal H_{\eta_i}(\mathcal P_\beta)$ in which he plays
only partial functions $u\in \mathcal P_\beta$ with $\min (\dom (u)) >m$. Let $u_i$ 
be the reply of $\tau_1$ if Player I starts by playing $\eta_i$ in $\mathcal H_{\eta_i}(\mathcal P_\beta)$. 
Then Player II plays $v_i = \{ (m,n_1)\} \cup u_i$ in the current game. As before, we have that $v_i \in \mathcal P_\alpha$. 
Proceeding in this way, Player {\rm II} plays pairwise incompatible members of $\mathcal P_\alpha$ as long
as the game last. Thus Player {\rm II} has a winning strategy in $\mathcal H_{\xi}(\mathcal P_\alpha)$, as desired. 
	 
 We now show that Player {\rm I} has a winning strategy in $\mathcal H_{\omega^\alpha}(\mathcal P_\alpha)$. 
 Of course, Player {\rm I} starts by playing $\omega^\alpha$. Suppose Player {\rm II} responds by playing some 
 $u_0\in \mathcal P_\alpha$. Fix an integer $m$ such that $\dom(u_0)\subseteq m$ and let $\beta$ be the immediate $<_m$-predecessor of $\alpha$. First note that if $u\in \mathcal P_\alpha$ is incompatible with 
$u_0$ then $\dom (u)\cap m \neq \emptyset$ and $\dom (u)\setminus m\in \mathcal P_\beta$. 
Let $\mathcal D$ be the set of all $s\in \mathcal P_\alpha$ which are nonempty and such that $\dom (s) \subseteq m$. 
Note that $\mathcal D$ is finite. Let $t$ be the cardinality of $\mathcal D$ and let $\{ s_0,\ldots, s_{t-1}\}$ 
be an enumeration of $\mathcal D$. By the inductive assumption, there is a winning strategy, say $\tau$, for 
Player {\rm I} in $\mathcal H_{\omega^\beta}(\mathcal P_\beta)$. On the side, Player {\rm I}  starts $t$ runs of 
$\mathcal H_{\omega^\beta}(\mathcal P_\beta)$ simultaneously in which he simulates the moves of Player {\rm II} and
uses the responses of $\tau$ in order to produce a move in ${\mathcal H}_{\omega^\alpha}(\mathcal P_\alpha)$.
We may assume that the first move of $\tau$ is $\omega^\beta$. Player {\rm I} then plays 
$\omega^\beta \cdot t$ in ${\mathcal H}_{\omega^\alpha}(\mathcal P_\alpha)$.  At stage $i$ suppose Player {\rm II} plays 
$u_i\in \mathcal P_\alpha$ that is incompatible with the $u_j$, for $j<i$. 
In particular,  $u_i$ is incompatible with $u_0$ and hence $\dom (u) \cap m \neq \emptyset$. 
Let $\xi_{k,i-1}$ be the latest ordinal played by $\tau$ in the $k$-th run of $\mathcal H_{\omega^\beta}(\mathcal P_\beta)$.
Let $r<t$ be such that $u_i \restriction m = s_{r}$. Player {\rm I} then considers the $r$-th
run of $\mathcal H_{\omega^\beta}(\mathcal P_\beta)$ and simulates a move of Player {\rm II}
in that game by playing $u_i \restriction [m,\omega)$. Note that $u_i \restriction [m,\omega)\in \mathcal P_\beta$ 
and is incompatible with 	 $u_j \restriction [m,\omega)$,
for all $j<i$ such that $u_j\restriction m =s_r$. Thus, $u_i \restriction [m,\omega)$
is a legitimate move by Player {\rm II} in that position of $\mathcal H_{\omega^\beta}(\mathcal P_\beta)$.
Let $\xi_{r,i}$ be the response of $\tau$. For all $k\neq r$, Player {\rm I} considers that no move is made in the $k$-th copy 
of $\mathcal H_{\omega^\beta}(\mathcal P_\beta)$ and sets 
$\xi_{k,i}=\xi_{k,i-1}$. Finally, in $\mathcal H_{\omega^\alpha}(\mathcal P_\alpha)$, 
Player {\rm I} plays
	 \[ 
	 \xi_i = \xi_{0,i} \oplus \xi_{1,i} \oplus \ldots \oplus \xi_{t-1,i}.
	 \]
	 Since $\xi_{r,i}<\xi_{r,i-1}$ and $\xi_{k,i}=\xi_{k,i-1}$, for all $k\neq r$, it follows
	 that $\xi_i < \xi_{i-1}$. Since $\tau$ is a winning strategy for Player {\rm I} in 
$\mathcal H_{\omega^\beta}(\mathcal P_\beta)$, it follows that Player {\rm I} can continue playing in this way as long
as Player {\rm II} plays pairwise incompatible members of $\mathcal P_\alpha$. Hence, this is a winning strategy for 
Player {\rm I} in $\mathcal H_{\omega^\alpha}(\mathcal P_\alpha)$, as required.   
\end{proof}
 
We shall need a version of the following lemma due to Roberts \cite{Ro:Maharam}. 

\begin {lemma}[Roberts' Selection Lemma]\label{Roberts} Let $\Vert \cdot \Vert$ be an admissible norm.
Suppose $s,t$ are integers and $I_l$ is a finite subset of $\mathbb N$ with $\Vert I_l\Vert \geq st$,
for all $l <s$. Then there is a permutation $\pi$ of $\{ 0,\ldots s-1\}$ and sets $J_i\subseteq I_{\pi(i)}$, for all $i<s$, such that  $J_1<J_2<\dots <J_s$ and $\Vert J_i\Vert =t$, for all $i<s$.

\end{lemma}
\begin{proof}
We essentially repeat the original argument. 
We define integers $k_i$ and $\pi(i)$, with $\pi(i)<s$, and sets $J_i\subseteq I_{\pi(i)}$ by induction on $i<s$. 
Ty begin, by (1) of Definition~\ref{admissible-norms}, we can find the least integer $k_0$ such that 
$\Vert I_l \cap k_0\Vert=t$, for some $l<s$. We let $\pi(0)$ be the least such $l$
and let $J_0=I_{\pi(0)} \cap k_0$. Note that, again by (1) of Definition~\ref{admissible-norms},
$\Vert I_l \setminus k_0 \Vert \geq (s-1)t$, for all $i\neq \pi(0)$. 
Having defined $k_j$ and $\pi(j)$, for all $j<i$,  
let ${k_{i}}$ be the least integer such that $\Vert I_l \cap [k_{i-1},k_i) \Vert \geq t$,
for some $l\neq \pi(0),\ldots, \pi(i-1)$. Let $\pi(i)$ be the least such $l$ and let 
$J_i= I_{\pi(i)}\cap [k_{i-1},k_i)$. We can clearly continue the construction for all $i<s$.
\end{proof}

We shall also need the following simple lemma which is the main reason why we require
our admissible norms to be spreading. 

\begin{lemma}\label{zapping} Let $\| \cdot \|$ be an admissible norm. Suppose  $C,D\subseteq \mathbb N$ 
are such that  $\| C \|\geq 3$ and $\| D\|=1$. Then there are consecutive elements
$c,d$ of $C$ such that $[c,d)\cap D=\emptyset$. 
\end{lemma}

\begin{proof}  Let $t$ be the cardinality of $C$ and let $\{ c_i : i < t\}$ be the increasing enumeration of $C$. 
Suppose $D \cap [c_i,c_{i+1})\neq \emptyset$, for all $i<t-1$, and pick $e_i \in D \cap [c_i,c_{i+1})$, for all $i< t-1$. 
Let $C'=C\setminus \{ c_0\}$ and $E=\{ e_i: i <t-1\}$.
Since $E\subseteq D$ and $\| D\|=1$ it follows that $\| E \| \leq 1$. Since $E\leq_s C'$ and $\| \cdot \|$ 
is spreading we also get that $\| C'\|\leq 1$. Now, $\| \{ c_0\} \|=1$, hence, by subadditivity of $\| \cdot \|$
we get that $\| C \| \leq 2$, a contradiction.
\end{proof}

\section{Talagrand's construction revisited}

In this section we associate to each admissible norm $\Vert \cdot \Vert$ an exhaustive submeasure
on a countable atomless Boolean algebra. The construction generalizes the one of Talagrand \cite{Talagrand},
which itself builds on previous work of Roberts \cite{Ro:Maharam} and Farah  \cite{Fa:Examples}.
In our case  special care has to be taken  in order  to take into account the fact that $\Vert \cdot \Vert$ 
is only subadditive rather than additive.
We start by describing the topological space and Boolean algebra that we will work with.

Let $T=\prod_n2^n$. For $n\in \mathbb N$, let $\mathcal{B}_n$ denote the algebra of subsets of $T$ that depend only on the coordinates  $< n$. Then $\mathcal{B}=\bigcup_n \mathcal{B}_n$ is the algebra of clopen subsets of $T$.
We denote by $\mathcal{A}_n$ the set of atoms of $\mathcal{B}_n$ and call them the atoms of rank $n$.
For $X\seq T$ we will write 
\[
[X]_n=\bigcap\{B\in \B_n:X\seq B\}=\bigcup \{A\in \A_n :A\cap X\neq \emptyset\}
\]
to describe the smallest clopen set in $\B_n$ containing $X$. We also write ${\rm int}_n(X)$
for the largest clopen set in $\B_n$ contained in $X$, i.e.
\[
{\rm int}_n(X) = \bigcup \{ A\in \mathcal A_n: A \subseteq X\}.
\]
Let us recall that $\mathcal P$ denotes the collection of all partial functions $u$ such that $\dom(u)\in [\mathbb N]^{<\omega}$ 
and $u(k)<2^k$, for all  $k\in \dom(u)$. If $u\in \mathcal P$ we let $N_u=\{ x\in T: u \subseteq x\}$. 
Then $\mathcal A_n$ is precisely the set of the $N_u$, for $u \in \mathcal P$ with	 $\dom(u)=n$. 

Fix, for the rest of this and the next section, an admissible norm $\Vert \cdot \Vert$. Our goal is
to define a positive exhaustive submeasure $\nu:\mathcal B\rightarrow [0,+\infty]$ 
such that $\nu(N_u)\geq 8$, for  all $u\in \mathcal P$ with $\Vert \dom(u) \Vert \leq 1$.
If $\Vert \cdot \Vert_{\alpha}$ is the admissible norm derived from the family $\mathscr S_\alpha$
from the previous section, by Lemma~\ref{Rank} we have that ${\rm rk}(\nu) \geq \omega^\alpha$. In
the last section we will also give an upper bound on ${\rm rk}(\nu)$.

In order to define our submeasures we will use  classes $\mathcal{F}$ of marked weighted sets, objects that have three components: the first one is a clopen subset of $T$, the second one is a finite set of coordinates and the third is a nonnegative real called the {\em weight} of the marked set. 
 
\begin{defi}
For $\mathcal{F}\subseteq \mathcal{B}\times \mathcal{P}(\mathbb N )\times \Rbb^+$,
\begin{align*}
X(\mathcal{F})&=\bigcup_{(X,I,w)\in \mathcal{F}} X,\text{ and }\\
w(\mathcal{F})&=\sum_{(X,I,w)\in \mathcal{F}} w.
\end{align*}
\end{defi}
We use the classes of marked weighted sets to define outer submeasures on $\B$. 
\begin{defi}
For $\mathcal{E}\subseteq \B\times \mathcal{P}(\mathbb N)\times \Rbb^+$, define 
$\phi_\mathcal{E}:\mathcal{B}\to \Rbb^+$ by setting
\[
\phi_{\mathcal{E}}(X)=\inf\{w(\mathcal{F}):\mathcal{F}\subseteq\mathcal{E} \text{ is finite and } X\subseteq X(\mathcal{F})\}
\]
By convention, we let $\phi_{\mathcal E}(\emptyset)=0$ and $\phi_{\mathcal E}(X)=+\infty$, for all
$X\in \B$ that is not covered by $X(\mathcal F)$, for any finite $\mathcal F \subseteq \mathcal E$. 

\end{defi}

The following notation will be frequently used. In particular, if $A\in \mathcal A_m$,
for some $m$, and $C\subseteq A$, we will use it to define the relative submeasure of $C$
inside $A$. 

\begin{defi}
Let $A\in \A_m$ and let $u\in \mathcal P$ be such that $A=N_u$.
Define $\pi_A:T\to A$ by
\[
\pi_A(z)(i) = \begin{cases} u(i),  &\mbox{if } i <m,  \\
                          z(i), & \mbox{otherwise. } \end{cases}
\]
\end{defi}

We now recall the definition of a thin set relative to a given submeasure. 
This notion was initially introduced by Farah in \cite{Fa:Examples} who used it to construct
examples of $\epsilon$-exhaustive pathological submeasures. It also plays a key role
in Talagrand's construction.

\begin{defi} Suppose $m<n$, $\phi :\mathcal{B}\rightarrow [0,+\infty]$ is a function and $X\in\mathcal{B}$, 
then $X$ is $(m,n,\phi )$-{\em thin} if for all 
$A\in \mathcal A_m$  there is $H\in \mathcal B_n$ such that $H\subseteq A\setminus X$ and $\phi(\pi^{-1}_A(H))> 1$.
For $I\subseteq \mathbb N$, $X$ is $(I,\phi)$-{\em thin} $I$, if it is $(m,n,\phi )$-thin, for
all $m,n\in I$ with $m<n$.
\end{defi}

Notice that $\pi_A^{-1}(H)$ is obtained by simply copying $H$ inside all other
atoms in $\mathcal A_m$. In all our cases $\phi$ will be a submeasure and we think
of $\phi(\pi^{-1}_A(H))$ as the submeasure of $H$ relative to $A$. 
Therefore, saying that $X$ is $(m,n,\phi )$-thin simply means that ${\rm int}_n(A\setminus X)$
is large, i.e. has submeasure bigger than  $1$ relative to $A$, for every $A\in \mathcal A_m$.

Before we present the details, let us describe the main ideas of the construction. 
We shall fix a sequence of positive reals $(a_k)_k$ such that $\sum_k a_k$ converges
and a sequence of integers $(M_k)_k$ quickly increasing to $+\infty$. 
We identify two properties of a submeasure $\phi$ which together imply 
that $\phi$ is exhaustive. 

\begin{defi}\label{thinness-property} Let $k$ be an integer. We say that a submeasure $\phi$ on $\mathcal B$ has
the $k$-{\em thinness property} if $\phi(X)\leq 2^{-k}$, 
for every $X\in \mathcal B$ which is $(I,\phi)$-thin for some set $I$ with $\| I \|=M_k$. 
\end{defi}

Note that this notion depends on our chosen norm  $\| \cdot \|$. If the norm is not clear
from the context we will explicitly specify it. The next definition is more technical, it expresses
a form of regularity of a submeasure $\phi$. It is motivated by the notion of a {\em potentially exhaustive}
submeasure introduced in \cite{Fa:Examples}.

\begin{defi}\label{covering-property} Let $\phi$ be  a submeasure on $\mathcal B$. 
Suppose $m$ is an integer and $E\in \mathcal B$ does not depend on coordinates $<m$
and $\phi(E)<2$. Let $n(E)$ be the least integer $n$ such that $E\in \mathcal B_n$. 
A sequence $\{C^m_r(E): m<r\leq n(E)\}$ is an $m$-{\em covering sequence} for $E$ if: 
\begin{enumerate}
\item[$(1)$] $C_r^m(E) \in \mathcal B_r$, for every $r$ such that $m < r\leq n(E)$, 
\item[$(2)$] ${\rm int}_j(E)\subseteq \bigcup \{ C_r^m(E): m< r \leq j\}$, for every $j$ such that $m < j\leq n(E)$,
\item[$(3)$] $\sum_{m<r\leq n(E)} \phi(C_r^m(E))\leq 4$. 
\end{enumerate}
\noindent We say that $\phi$ has the $m$-{\em covering property} if every such $E$ has an $m$-covering sequence. 
Finally, we say that $\phi$ has the {\em covering property} if it has the $m$-covering property, for every $m$.
\end{defi}

It will be fairly easy to show that if $\phi$ is a submeasure satisfying the covering and thinness
properties and such that $\phi(T)\geq 8$ then $\phi$ is exhaustive. 
In order to construct such $\phi$ the natural idea is to define a sequence
$(\mathcal F_k)_k$ of subsets of $\mathcal B \times {\mathcal P}(\mathbb N) \times \mathbb R^+$
as follows. Start with $\mathcal F_0=\emptyset$. Given $\mathcal F_k$ let $\phi_k= \phi_{\mathcal F_k}$.
Construct $\mathcal F_{k+1}$ by adding to $\mathcal F_k$ all triples $(X,I,w)$ such that 
$X$ is $(I,\phi_k)$-thin, $\| I \| \leq M_k$ and 
\[
w\geq 2^{-k}\Big(\frac {M_k}{\|I\|}\Big)^{a_k}.
\]
The reason for this last requirement is to ensure the covering property. Namely, suppose 
 $(X,I,w)\in  \mathcal F_{k+1}$. In some situations we will need to replace  $(X,I,w)$ 
 by another triple $(X',I',w')\in \mathcal F_{k+1}$, where
 $X'$ is  a superset of $X$ and depends only on coordinates in some interval $[m,n)$,
 $I'=I\cap [m,n]$,  and $w'$ is not too big relative to $w$. 
If $\|I'\|$ is not too much smaller than $\|I\|$, the fact that we have $a_k$ in the exponent
will allow us to choose $w'$ which is very close to $w$. 
This construction would ensure that $\phi_{k+1}$ satisfies the $k$-thinness condition
and, since the sequence $(\phi_k)_k$ is decreasing, this condition will remain to hold
for the later $\phi_l$. 
The problem with this scenario is that, in order to obtain an exhaustive submeasure,
we would have to continue this process for all $k$, but as explained in \cite{Fa:Examples},
the limit submeasure $\lim_k \phi_k$ collapses to $0$. 
 
 The main new idea \cite{Talagrand} is to reverse this process. Namely, for each 
 $p$ we define families $\mathcal C_{k,p}$, for $k\leq p$, by backwards induction. 
 We can start with $\mathcal C_{p,p}=\emptyset$. 
 Given $\mathcal C_{k+1,p}$ we let $\nu_{k+1,p}=\phi_{\mathcal C_{k+1,p}}$
 and we construct $\mathcal C_{k,p}$ by adding to  $\mathcal C_{k+1,p}$ all triples
 $(X,I,w)$ satisfying the thinness and the weight conditions relative to $\nu_{k+1,p}$.
 We have that the $\nu_{k,p}$ decrease as $k$ gets smaller, but we are able to
 guarantee that $\nu_{0,p}(T)\geq 8$. In this way we will have that $\nu_{k,p}$ satisfies
 the $l$-thinness property, for all $k \leq l <p$. Then we pick a non principal ultrafilter $\mathcal U$
 on $\mathbb N$ and let $\nu_k =\lim_{p\rightarrow \mathcal U} \nu_{k,p}$, for each $k$.
 The covering property and the $l$-thinness property are preserved by taking
 the $\mathcal U$-limit of submeasures, so the resulting $\nu_k$ will all be exhaustive.

 We now turn to the details of the construction. We define the sequences $(a_k)_k$ and
 $(M_k)_k$ as follows:
\begin{eqnarray*}
a_k&=&\frac{1}{(k+5)^3}\\
M_k &=& 2^{2k+12}\cdot 2^{(k+4)(k+5)^3}.\
\end{eqnarray*}

\begin{defi}\label{definition-submeasures}
Fix an integer $p\in\mathbb N$. We define families $\mathcal C_{k,p}$ for $k\leq p$, by downwards induction on $k$. 
Once we have $\mathcal C_{k,p}$ we let $\nu_{k,p}=\phi_{\mathcal C_{k,p}}$.
We start by letting $\mathcal C_{p,p}=\emptyset$. 
Suppose $k<p$ and $\mathcal{ C}_{k+1,p}$ has been defined. We let:
\begin{eqnarray*}
 \mathcal{ E}_{kp}
&=&\{(E,I,w):E\in\mathcal{B} ,\,I\subseteq\mathbb N,\,
\|I\|\le M_k,
w\ge 2^{-k}\Big(\frac {M_k}{\|I\|}\Big)^{a_k},\\
&\, &
E\text{ is }(I,\nu_{k+1,p})\text{-thin}\},\\
\mathcal{ C} _{k,p} &=& \mathcal{ C} _{k+1,p} \cup \mathcal {E}_{k,p}.
\end{eqnarray*}
We also define a sequence $(c_k)_k$ by setting $c_0=8$ and $c_{k+1}=4^{a_k}c_k$, for all $k$.
\end{defi}

Let us compare our construction with the one from \cite{Talagrand}. 
First, Talagrand starts by setting $\mathcal C_{p,p}=\mathcal D$, for some suitable family $\mathcal D$.
This was done in order to ensure that all the submeasure are pathological, but it is not really necessary
since we will have an explicit reason why our submeasures are not uniformly exhaustive. 
The main difference is that in the definition of the  $\mathcal E_{k,p}$, instead of the cardinality of $I$ 
we use $\| I\|$, where $\| \cdot \|$ is our given admissible norm. 
Of course,  the notion of an admissible norm was tailor-made so that analogs of the key arguments
from \cite{Talagrand} would go through. The upshot is that by varying our norm $\| \cdot \|$ we obtain uncountably many
essentially different examples of exhaustive non uniformly exhaustive submeasures. 
For the remainder of this section we prove some technical lemmas and show
that $\nu_{k,p}(N_u)\geq 8$, for all $k\leq p$ and all $u\in \mathcal P$ such that $\| \dom(u)\|\leq 1$. 


\begin{lemma}\label{duga}
Let $k,p$ and $m$ be integers with $k\leq p$. Suppose $(X,I,w)\in \mathcal{C}_{k,p}$ and $A\in \mathcal{A}_m$. 
Suppose that $n>m$ and $I'=I\cap [m,n)$ is non-empty. Set $X'=[\pi_A^{-1}(X\cap A)]_n$. 
Then $(X',I',w')\in \mathcal{C}_{k,p}$, where $w'=w\cdot (\frac{\|I\|}{\|I'\|})^{a_k}$.
\end{lemma}
\begin{proof} 
By the definition of $\mathcal{C}_{kp}$, there is some $r$ with $k\leq r <p$ such that $(X,I,w)\in \E_{rp}$. 
Let us fix such $r$ and let us show that $X'$ is $(I',\nu_{r+1,p})$-thin. 
Since we only used $X\cap A$ in the definition of $X'$, we may assume that $X\seq A$. 
 Let $i,j\in I'$ be such that $i<j$. We need to show that for every $A_1\in \A_i$ there is $H\in \B_j$
 such that $H\subseteq A_1 \setminus X'$ and $\nu_{r+1,p}(\pi_{A_1}^{-1}(H)) > 1$. 
 Now, if $A_1\subseteq A$  this follows from the fact that $X$ is $(i,j,\nu_{r+1,p})$-thin. 
 If $A_1\cap A = \emptyset$ let $A_2=\pi_A(A_1)$. Then, as before, we can find  $H\subseteq A_2 \setminus X$ 
 such that $H\in \B_j$ and $\nu_{r+1,p}(\pi_{A_2}^{-1}(H)) >1$. 
 Let $H'=\pi_{A_2}^{-1}(H)\cap A_1\in \B_j$. Since $\pi_{A_1}^{-1}(H')= \pi_{A_2}^{-1}(H)$ we also have  
 $\nu_{r+1,p}(\pi_{A_1}^{-1}(H')) > 1$. Moreover, we have $H'\cap \pi_{A}^{-1}(X)=\emptyset$. 
 Indeed, if $x\in H'\cap \pi_A^{-1}(X)$ we would have that   $\pi_{A_2}(x)\in X\cap H=\emptyset$.
 Since $H'\in \B_j$ we also have $H'\in \B_n$. Therefore, $H'\cap [\pi_{A}^{-1}(X)]_n=\emptyset$ i.e. $H'\cap X'=\emptyset$. 
Finally, let us check the weight condition. Since,  by definition, $w\ge 2^{-r}\bigr(\frac {M_r}{\|I\|}\bigr)^{a_r}$ and $k\le r$, we get 
\begin{align*}
w'&=w\cdot \Bigr(\frac{\|I\|}{\|I'\|}\Bigr)^{a_k}\\
&\ge w\cdot \Bigr(\frac{\|I\|}{\|I'\|}\Bigr)^{a_r}\\
& \ge 2^{-r}\Bigr(\frac {M_r}{\|I'\|}\Bigr)^{a_r}.
\end{align*}
This proves that $(X',I',w')\in \E_{rp}\seq \mathcal{C}_{rp}\seq \mathcal{C}_{kp}$, as desired.
\end{proof}

For the purpose of the following lemma we shall extend our previous notation and 
if $\mathcal D$ is a subset of $[\mathbb N]^{<\omega}\times \mathbb R^+$ we shall write
\[
w(\mathcal{D})=\sum_{(I,w)\in \mathcal{D}} w.
\]

\begin{lemma}\label{vazna}
Suppose $t\geq 5$ and $J_0< \dots <J_{s-1}$ are finite sets with $\Vert J_i\Vert \geq t$, for all $i<s$. 
Suppose $\mathcal F$ is a finite subset of $[\mathbb N]^{<\omega}\times \mathbb R^+$ and 
let $a=w(\mathcal F)$. Then we can find integers $m_i,n_i\in J_i$, with $m_i<n_i$ for all $i<s$, such that:
\begin{enumerate}
	\item $\Vert J_i \cap [m_i,n_i)\Vert \geq 3$, for all $i<s$,
	\item if we let $W=\bigcup_{i<s} [m_i,n_i)$ and 
	$\mathcal D=\{ (I,w)\in \mathcal F: \Vert I\setminus W \Vert < \frac{1}{2}\Vert I\Vert \}$
	then 
	\[
	w(\mathcal D)\leq \frac{3a}{t-4}.
	\]
	
\end{enumerate}

\end{lemma}
\begin{proof}
Let $c=\lfloor \frac{t-1}{3}\rfloor$.  For each $i<s$, we pick an increasing sequence $\{ m_{i,l}: l \leq c\}$ of elements
of $J_i$ such that $\Vert J_i \cap [m_{i,l},m_{i,l+1})\Vert =3$, for all $l <c$. 
Given $l <c$, let $W_l= \bigcup_{i<s} [m_{i,l},m_{i,l+1})$ and let 
\[
\mathcal D_l = \{ (I,w)\in \mathcal F : \Vert I \setminus W_l \Vert < \frac{\Vert I \Vert}{2} \}.
\]
Since the $W_l$ are pairwise disjoint and $\Vert \cdot \Vert$ is subadditive, it follows that 
the $\mathcal D_l$ are pairwise disjoint.  Therefore, we get that
\[
w(\mathcal{D}_0)+\dots  +w(\mathcal{D}_{c-1})\le  w(\FF)=a.
\]
By using the fact that the minimum of a finite sequence is less than or equal to its average, we conclude that 
there exists $l< c$ such that:
\[
w(\mathcal D_l)\le \frac{a}{c}\le \frac{a}{(t-1)/3-1}=\frac{3a}{t-4}.
\]
Therefore, we can let $m_i= m_{i,l}$ and $n_i=n_{i,l}$, for all $i<s$, and this satisfies
the conclusion of the lemma. 
\end{proof}

In the next proposition we adapt the argument of Theorem 5.1 from \cite{Talagrand}. 

\begin{prop}\label{lower bound} 
Suppose $k$ and $p$ are integers with $k\leq p$. Then $\nu _{k,p}(N_u)\ge c_k$, for every $u\in \mathcal P$
with $\| \dom (u) \|\leq 1$.
In particular, $\nu_{k,p}(T)\geq c_k$. 
\end{prop}
\begin{proof} Let us fix $p$ and prove the statement by backwards induction on $k$. 
If $k=p$ the statement is obvious since $\nu_{p,p}(X)=+\infty$, for every non empty $X\in \mathcal B$. 
Thus, let us assume $k<p$, the inequality holds for $k+1$, and let us check that it holds for $k$. 
Fix $u\in \mathcal P$ with $\| \dom (u) \| =1$ and a finite $\mathcal F \subseteq \mathcal C_{k,p}$ with $w(\mathcal F) <c_k$. 
We have to show that $N_u\nsubseteq X(\mathcal F)$. 

To begin let us fix $\mathcal{F}_1\seq \mathcal{E}_{k,p}$, $\mathcal{F}_2\seq\mathcal{C}_{k+1,p}$ 
such that $\mathcal{F}=\mathcal{F}_1\cup \mathcal{F}_2$.
Let $s=|\mathcal{F}_1|$. Since for $(X,I,w)\in \mathcal{F}_1$ we have $w\ge 2^{-k}$, we get $s\le 2^kc_k\le 2^{k+4}$. 
On the other hand, for every such $(X,I,w)$ we have
\[
2^{-k}\Big(\frac {M_k}{\|I\|}\Big)^{a_k} \!  \! \le w\le c_k\le 2^4,
\]
so an easy calculation gives us $\|I\| \ge 2^{k+8}s$.
Using Lemma \ref{Roberts}, we can enumerate $\mathcal{F}_1$ as $\{ (X_l,I_l,w_l) : l <s \}$, 
and find sets $J_0,\dots J_{s-1}$, such that  $J_l\subseteq I_l$ and $\|J_l\|\ge 2^{k+8}$, for every $l<s$,
and moreover  $J_0<\dots < J_{s-1}$.
Applying Lemma \ref{vazna}, where $a=w(\mathcal{F}_2)\le 2^4$, we can find  $m'_l,n'_l\in J_l$
with $m'_l <n'_l$, for $l < s$,   so that $\|J_l\cap [m'_l,n'_l)\|=3$ and,  if 
we let,
\[W'=\bigcup_{l<s}\; [m'_l,n'_l) \mbox{ and } 
\mathcal F'_3= \{ (X,I,w) \in \mathcal F_2: \| I \setminus W'\| < \frac{1}{2}\| I\| \},
\]
then 
\[
w(\mathcal{F}'_3) \le \frac{3\cdot 2^4}{2^{k+8}-4}\le \frac{12}{2^{k+6}-1}\le \frac{1}{4}.
\]
Since $\| \dom (u)\| =1$, by Lemma \ref{zapping} we can find consecutive elements
$m_l,n_l\in J_l \cap [m'_l,n'_l)$ such that $\dom (u) \cap [m_l,n_l)=\emptyset$, for every $l<s$.
Now, let 

\[
W=\bigcup_{i < s} \;  [m_l,n_l), \hspace{5mm}
\mathcal{F}_3=\{(X,I,w)\in \mathcal{F}_2:\|I\setminus W\| < \frac{1}{2}\|I\|\}, \hspace{5mm}
 \mathcal{F}_4= \mathcal{F}_2\setminus \mathcal{F}_3.
 \]
 Since $\mathcal F_3\subseteq \mathcal F'_3$ we have that $w(\mathcal F_3)\leq w(\mathcal F'_3)\leq 1/4$. 
Since $I=(I\cap W)\cup (I\setminus W)$, for $(X,I,w)\in \mathcal{F}_3$, we also get 
$\|I\cap W\| \geq  \frac{1}{2}\|I\|$.
For $l <s$ let

\[
\mathcal{F}_{3,l}=\{(X,I,w)\in \mathcal{F}_3:\|I\cap [m_l,n_l)\| \geq 2^{-k-5}\|I\|\}.
\]
Since $s\le 2^{k+4}$ and $\|I\cap W\| \ge \frac{1}{2}\|I\|$, 
we get $\mathcal{F}_3=\bigcup\limits_{l < s}\mathcal{F}_{3,l}$.

\begin{cla}\label{atoms} For every $l< s$ and  $A\in \mathcal A_{m_l}$,
there is $A'\in \mathcal A_{n_l}$ with  $A'\subseteq A\setminus (X_l\cup X(\FF_{3,l}))$.
\end{cla}

\begin{proof}
Let $A\in \mathcal{A}_{m_l}$. 
Using Lemma \ref{duga} for $C=[\pi_A^{-1}(A\cap X(\FF_{3,l}))]_{n_l}\in \B_{n_l}$ we get:
\[
\nu_{k+1,p}(C)\le 2w(\FF_{3,l})\le \frac{1}{2}.
\]
Since $(X_l,I_l,w_l)\in \E_{k,p}$, $X_l$ is $(m_l,n_l,\nu_{k+1,p})$-thin. 
Therefore, for $K=T\setminus [\pi_A^{-1}(A\cap X_l)]_{n_l}$, 
we have $\nu_{k+1,p}(K) > 1$. 
Since both $C$ and $K$ belong to $\B_{n_l}$, there is an atom $A'\in \mathcal{A}_{n_l}$, 
such that $A'\subseteq K\setminus C$. Since $\pi_A^{-1}(K)=K$
and $\pi_A^{-1}(C)=C$, we may find such $A'$ which is contained in $A$. 
\end{proof}

Now define a function $\Gamma : T\to T$ as follows. 
For $x\in T$ and $j\in \mathbb N \setminus W$, we let 
$\Gamma (x)(j)=x(j)$. Let us consider some $x\in T$ 
and an interval $[m_l,n_l)$ and suppose $\Gamma (x)\restriction m_l$
has been defined.  Let $v= \Gamma (x)\restriction m_l$. Applying Claim \ref{atoms} to
$A=N_v$, we find some $A'\in \mathcal A_{n_l}$ with $A'\seq A \setminus (X_l \cup X(\FF_{3,l}))$.
Let $v'$ be such that $A'=N_{v'}$. We then let $\Gamma (x) \restriction n_l=v'$. 
Notice that we have assured that $\Gamma (T)$ is disjoint from $X(\mathcal F_1)\cup X(\mathcal F_3)$. 

\begin{cla}\label{F_4} $\nu_{k+1,p}(\Gamma^{-1}[X(\FF_4)])<c_{k+1}$.
\end{cla}

\begin{proof}
Let $(X,I,w)\in \FF_4$. We want to estimate $\nu _{k+1,p}(\Gamma^{-1}[X])$. 
There is an $r$ such that $k+1\le r<p$ and $(X,I,w)\in \E_{r,p}$. 
We will show first that for $m,n\in I$, $m<n$, if $[m,n)$ is disjoint from $W$, 
then $\Gamma^{-1}[X]$ is $(m,n,\nu _{r+1,p})$-thin. 
In order to see this, let $A\in \A_m$, and let $A'$ be an atom in $\A_m$ 
such that $\Gamma (A)\subseteq A'$. 
Since  $X$ is $(m,n,\nu_{r+1,p})$-thin, within $A'$
we replicate the thinness of $X$ in $A'$ inside $A$ to establish the thinness of $\Gamma^{-1}[X]$. 
More precisely, since $X$ is $(m,n,\nu_{r+1,p})$-thin, within $A'$ there exists   
$H\in \mathcal{B}_n$ such that $H\seq A'$, $H\cap X=\emptyset$. 
and $\nu_{r+1,p} (\pi_{A'}^{-1}(H)) > 1$. Since $H$ is disjoint from $X$, 
we also have that $\Gamma^{-1}[H]$ is disjoint  from $\Gamma^{-1}[X]$. 
From $\Gamma(\pi_A(\pi_{A'}^{-1}(H))\seq H$, we get  

\[
\pi_{A'}^{-1}(H)\seq \pi_{A}^{-1}(\Gamma^{-1}[H])\seq  \pi_A^{-1}(H).
\]

\noindent We assumed  that $\nu_{r+1,p} (\pi_{A'}^{-1}(H)) >1$, 
so $\nu_{r+1,p} (\pi_A^{-1}(H)) > 1$, which proves that 
$\Gamma^{-1}[X]$ is $(m,n,\nu _{r+1,p})$-thin.

We now have to deal with pairs of elements of $I$ that are separated by $W$. 
Since thinness is monotone in the second coordinate, this is a problem only for $m,n\in I$ with $m<n$
such that $m$ is the  last element in $I$ preceding some interval $[m_l,n_l)$. 
We saw in the beginning of the proof that $\|I\| \geq 2^{k+8}s\geq 4s$, so $s\le \|I\|/4$. 
From the definition of $\FF_4$ we have 
 $\|I\setminus W\|\geq \|I\|/2$. For every $l< s$, let $i_l$ be the largest element of $I$ 
 below $m_l$. Then, for:
\[
I'=I\setminus (W\cup \{i_l:l< s\}),
\]
we have that  $\|I'\| \geq  \|I\|/4$. 
We now have that $\Gamma^{-1}[X]$ is  $(I',\nu _{r+1,p})$-thin. 
We also need a bound for the norm. Let 
\[
w'=w  \Big(\frac{\|I\|}{\|I'\|}\Big)^{a_k} \! \!  \le 4^{a_r} w\leq 4^{a_k} w.
\]
We have that $(\Gamma^{-1}[X],I',w')\in\E_{r,p}\subseteq \mathcal C_{k+1,p}$. 
This establishes $\nu _{k+1,p} (\Gamma^{-1}[X])\le 4^{a_k}w$.
Now we have,
$$
\nu _{k+1,p}(\Gamma^{-1}[X(\FF_4)])\le 4^{a_k}w(\FF_4)<4^{a_k}c_k=c_{k+1},
$$
as needed.
\end{proof}

Now, by Claim \ref{F_4} and the inductive assumption,
pick some $z\in N_u \setminus  \Gamma^{-1}[X(\FF_4)]$.
 Since $W\cap \dom (u)=\emptyset$ and $\Gamma (x)(j)=x(j)$,
 for every $x\in T$ and $j \notin W$, it follows that $\Gamma (z)\in N_u$.
 On the other hand, we have already shown that 
 $\Gamma (T)$ is disjoint from $X(\FF_1)$ and $X(\FF_3)$. 
 Since $z\in N_u \setminus  \Gamma^{-1}[X(\FF_4)]$ it follows that 
 \[
 \Gamma (z) \notin X(\FF_1)\cup X(\FF_3) \cup X(\FF_4) = X(\FF).
 \]
 This completes the proof of Proposition \ref{lower bound}.
 \end{proof}

\begin{defi} Let $\mathcal{U}$ be a non principal ultrafilter on $\mathbb N$. For $k\in \mathbb N$
and $E\in \mathcal B$, we define $\nu_k(E)=\lim_{p\rightarrow \mathcal{U}}\nu_{k,p}(E)$.
 We write $\nu$ for $\nu_0$. 
\end{defi}

\begin{prop}
For every integer $k$, $\nu_k$ is a submeasure and $\nu_k(N_u)\ge 8$,
for every $u\in \mathcal P$ with  $\| \dom (u)\| =1$. In particular, $\nu_k$
is not uniformly exhaustive. 
\end{prop}

\begin{proof}
First note that $\nu_k$ is a submeasure as an ultrafilter limit of submeasures. 
Let $u \in \mathcal P$ be such that $\| \dom (u)\| =1$.
By Proposition \ref{lower bound},  $\nu_{k,p} (N_u)\ge c_k\ge 8$, for every $p\geq k$.
Therefore $\nu_k(N_u)\ge 8$, as well. 
Given an integer $n$ and $i<2^n$, let  $u_i=\{(n,i)\}$. 
Since $\| \{ n\} \| =1$, we have that $\nu_k(N_{u_i})\geq 8$, for all $i < 2^n$. 
Since the family $\{ N_{u_i}: i < 2^n\}$ is pairwise disjoint, for every $n$, it follows
that $\nu_k$ is not $8$-uniformly exhaustive. 
\end{proof}

For a countable ordinal $\alpha >0$, let us write $\nu^\alpha$ for the submeasure $\nu$ constructed from 
the admissible norm $\| \cdot \|_\alpha$ from Definition \ref{alpha-norm}.
By  Lemma \ref{Rank} and Proposition \ref{lower bound}, we have the following immediate
corollary. 

\begin{coro}\label{exhaustivity-rank-Schreier} The exhaustivity rank of $\nu^\alpha$
is at least $\omega^\alpha$, for $0 < \alpha <\omega_1$. 
\qed
\end{coro}

\section{Exhaustivity}

In this section we still work with a given admissible norm $\| \cdot \|$ 
and the submeasures $\nu_{k,p}$ given by Definition \ref{definition-submeasures}.
We now turn to the proof that the limit submeasures $\nu_k$ are exhaustive. 
We organize our argument in  a way to also be able  to provide upper bounds on 
their exhaustivity ranks.

\begin{lemma}\label{level-covers}
For every $k$ and $p$ with $k\leq p$, the submeasure $\nu_{k,p}$
has the covering property. 
\end{lemma}

\begin{proof} Suppose $m$ is an integer, $E\in \mathcal B$ does not depend
on coordinates $<m$ and $\nu_{k,p}(E)<2$. Let $n=n(E)$. 
We need to construct an $m$-covering sequence for $E$. 
Fix some $\mathcal F\subseteq \mathcal C_{k,p}$ with 
$E\subseteq X(\mathcal F)$ and $w(\mathcal F) < 2$. 
For $r>m$ we let:
\[
\mathcal F_{r}=\{(X,I,w)\in\mathcal F : \| I\cap [m,r)\|<\frac{1}{2}\|I\| \mbox{ \&  }
\|I\cap [m,r]\| \ge\frac{1}{2}\|I\| \}.
\]
We also let
\[
\mathcal F'=\{(X,I,w)\in\mathcal F: \|I\cap m\| \ge \frac{1}{4}\|I\|\}.
\]
We use Lemma \ref{duga} to get a set $B \in \mathcal B_m$ such that 
$X(\mathcal F')\subseteq B$ and 
\[
\nu_{k,p}(B)\le 4^{a_k}w(\mathcal F')\leq 4.
\]
Since $\nu_{k,p} (T)\ge 8$, $B\neq T$, so there exists $A_m\in \mathcal A_m$ 
such that $A_m\cap X(\mathcal F')=\emptyset$.
For $(X,I,w)\in  \mathcal F_{r}$, let $X'= [\pi^{-1}_{A_{p}}(X\cap A_m)]_r$
and $I'=I\cap[m,r]$. Note that $X\cap A_m\subseteq X'$. 
By Lemma \ref{duga} again, we can find some $w'\leq 2w$ such that
$(X',I',w')\in \mathcal C_{k,p}$.
 Let $ \mathcal F'_r$ be the collection of triples
$(X',I',w')$ obtained in this way. 

\begin{cla}\label{new-primes} For every $j$ such that $m<j\leq n$, we have
\[
{\rm int}_j(E)\subseteq \bigcup_{m<r\le j} X(\mathcal F'_r).
\]
\end{cla}

\begin{proof} Note that  ${\rm int}_j(E)$  and the sets $X(\mathcal F'_{r})$, for $m< r\leq j$, depend
only on the coordinates in the interval $[m,j)$. Therefore, if the inclusion does not hold we can 
find $A\in \mathcal A_j$ such that $A\subseteq A_m \cap E$, yet $A\cap X(\mathcal F'_r)=\emptyset$,
for all $r$ with $m<r\leq j$. Since $A_p\cap X(\mathcal F_r)\subseteq A_{p}\cap X(\mathcal F'_r)$,
it follows that $A\cap X(\mathcal F_r)=\emptyset$, for every $r$ with $m<r\leq j$. Finally, we have that 
$A_m\cap X(\mathcal F')=\emptyset$. All this means that $A\subseteq X(\mathcal F'')$,
where 
\[
\mathcal F''= \mathcal F \setminus (\mathcal F' \cup \bigcup_{r\leq j}\mathcal F_r).
\]
Note that if $(X,I,w)\in \mathcal F''$ then $\| I \setminus j\| \geq \frac{1}{4}\| I \|$. 
By applying Lemma \ref{duga} one more time, we can find a set $X^*$, covering $X$ and depending
only on coordinates $\geq j$, and some $w^*\leq 2w$ such that, letting $I^*=I\setminus j$,
we have $(X^*,I^*,w^*)\in \mathcal C_{k,p}$. Let $\mathcal F^*$ be the set of
triples $(X^*,I^*,w^*)$ obtained in this way. Then $X(\mathcal F^*)$ depends
only on coordinates $\geq j$ and contains $X(\mathcal F'')$, and 
\[
w(\mathcal F^*)\leq 2w(\mathcal F)\leq 4.
\]
Since $A\subseteq X(\mathcal F^*)$ and $A\in \mathcal B_j$, 
it follows that $X(\mathcal F^*)$ covers all of $T$.  This implies that $\nu_{k,p}(T)\leq 4$,
a contradiction.
\end{proof}

\noindent For $m < r \leq n$, the set $X(\mathcal F'_r)$ depends only on coordinates in the interval $[m,r)$ and
\[
\nu_{k,p}(X(\mathcal F'_r))\leq w(\mathcal F'_r)\leq 2w (\mathcal F_r).
\] 
Therefore, we get that: 
\[
\sum_{m<r\leq n} \nu_{k,p}(X(\mathcal F'_r)) \leq \sum_{m<r\leq n}2w(\mathcal F_r)\leq 2w(\mathcal F) \leq 4. 
\]
It follows that if we let  $C_r^m(E)= X(\mathcal F'_r)$, for  $m<r\leq n$,
the resulting sequence is an $m$-covering sequence for $E$. 
\end{proof}

\begin{lemma}\label{limit-covers} The submeasure $\nu_k$ has the covering property, for every $k$. 
\end{lemma}

\begin{proof} Suppose $m$ is an integer, $E$ is a set in $\mathcal B$ that does not depend 
on coordinates $<m$, and $\nu_k(E) <2$. Let $n=n(E)$. By the definition of $\nu_k$, the set
$U =\{ p : \nu_{k,p}(E) < 2\}$ belongs to $\mathcal U$.
For each $p\in U$, fix an $m$-covering sequence $\{ C^m_{r,p}(E): m <r \leq n\}$ of $E$ with respect to $\nu_{k,p}$. 
Since $\mathcal U$ is a ultrafilters and the set of all possible such sequences is finite,
there is a fixed sequence $\{ C^m_r(E) : m < r \leq n\}$ such that
\[
V= \{ p\in U: C^m_{r,p}(E)= C^m_r(E), \mbox{ for all } m < r\leq n\} \in \mathcal U.
\]
It is clear now that  $\{ C^m_r(E): m <r \leq n\}$ is an $m$-covering sequence for $E$
with respect to $\nu_k$. 
\end{proof}

\begin{lemma}\label{limit-bound}
Let $k$ be an integer and $(E_i)_i$  a sequence of sets in ${\mathcal B}$ 
not depending on coordinates  $<m$ such that 	$\nu_k(\bigcup_{i< n}E_i) < 2$, for every $n$.
Then, for every $\eta >0$, there is  $C\in \mathcal B$ that does not depend on coordinates $<m$ such that
$\nu_k(C) \leq 4$ and $\nu_k(E_i\setminus C)\leq \eta$, for all $i$.
\end{lemma}

\begin{proof} For each $n$, let $G_n=\bigcup_{i<n}E_i$ and let $\vec{C}(G_n)$ be
an $m$-covering sequence for $G_n$. Since, for each $l>m$, there are only finitely many possibilities
for $\vec{C}(G_n)\restriction [m,l)$,  by K\"{o}nig's Lemma 
there is an infinite sequence $\vec{C}=\{ C_j:  j > m \}$ such that, for every $l>m$, there
are arbitrary large $n$ such that $\vec{C}\restriction l = \vec{C}(G_n)\restriction l$.
It follows that 
\[
 \bigcup_iE_i \subseteq \bigcup_{m < j}C_j \mbox{ and } \sum_{m<j}\nu_k(C_j)\leq 4.
\] 
 Let $l$ be such that $\sum_{l\leq j} \nu_k(C_j) \leq \eta$.
Then the set $C=\bigcup \{ C_j: m < j < l\}$ satisfies the conclusion of the lemma. 
\end{proof}

\begin{lemma}\label{finding-thinness}
Let $k$ and $m$ be integers and $\eta >0$. 
Suppose $(E_i)_i$ is a pairwise disjoint sequence of sets in $\mathcal B$. Then there is $n >m$
and $B\in \mathcal B_n$, $B$ is $(m,n,\nu_k)$-thin, and
$\limsup_{i\rightarrow \infty}\nu_k(E_i\setminus B)\leq \eta$.
\end{lemma}

\begin{proof}
Let $\eta'=\eta/| \mathcal A_m|$. For all $A\in \mathcal A_m$ we define
$H(A)\subseteq A$ such that 
\[ 
\nu_k(\pi_A^{-1}(H(A))) >1 \mbox{ and }
\limsup_{i\rightarrow \infty}\nu_k(E_i\cap  H(A))\leq \eta'.
\]
		
\noindent {\em Case 1.} There exists an integer $r$ such that $\nu_k (\pi_A^{-1}(A\cap \bigcup_{i< r}E_i))>1$.
\medskip

\noindent We let  $H(A)=A\cap \bigcup_{i< r}E_i$. Note that $E_i\cap H(A)=\emptyset$,
for all $i\geq r$.
\medskip

\noindent {\it Case 2.} $\nu_k (\pi_A^{-1}(A\cap \bigcup_{i< r}E_i))\leq 1$, for every integer $r$.

\medskip

\noindent Since the sets $\pi_A^{-1}(E_i)$ do not depend on coordinates $<m$, 
by Lemma \ref{limit-bound}, we can find $C\in \mathcal B$ that does not depend
on coordinates $<m$ such that $\nu_k (C)\le 4$ and 
\[
\limsup_{i\rightarrow \infty}\nu_k(\pi_A^{-1}(E_i)\setminus C)\leq \eta'.
\]
If we take $H(A)= A\setminus C$, then $C=T\setminus \pi_A^{-1}(H(A))$. 
Since $\nu_k(T)\geq 8$, we have  $\nu_k (\pi_A^{-1} (H(A))) >1$. 
Since $\pi_A$ is the identity on $A$, we have that, for all $i$,
\[
 E_i\cap H(A)  \subseteq \pi_A^{-1}(E_i)\setminus C,
\]
and so
\[
\limsup_{i\rightarrow \infty} \nu_k(E_i\cap H(A)) \leq 
\limsup_{i \rightarrow \infty} \nu_k(\pi_A^{-1}(E_i)\setminus C)\leq \eta'.
\]

\noindent Now, let $D(A)=A\setminus H(A)$, let $B=\bigcup \{D(A): A\in \mathcal A_m\}$,
and let $n$ be the least such that $B\in \mathcal B_n$.
Then $B$ and $n$ are as required. 
\end{proof}

\begin{lemma}\label{limit-thinness} The submeasure $\nu_k$
has the $s$-thinness property, for all $s\geq k$. 
\end{lemma}

\begin{proof} 
Fix some $s\geq k$, and suppose $\| I \| =M_s$ and $X$ is $(I,\nu_k)$-thin. 
Since this property of $X$ depends on the fact that the submeasure of finitely many sets is $>1$ 
and $\mathcal U$ is non principal,  it follows that:
\[
U = \{  p \geq s+1 : X \mbox{ is } (I,\nu_{k,p}) \mbox{-thin}                      \} \in \mathcal U.
\]
Fix some $p\in U$. Since $\nu_{k,p} \leq \nu_{s+1,p}$, we have
that $X$ is also $(I,\nu_{s+1,p})$-thin, and hence $(X,I,2^{-s})\in \mathcal E_{s,p}$.
It follows that  $\nu_{k,p}(X)\leq 2^{-s}$. Since this holds for all $p\in U$
and $\nu_k = \lim_{p \rightarrow \mathcal U}\nu_{k,p}$, we conclude that $\nu_k(X)\leq 2^{-s}$, as desired.
\end{proof}

\begin{prop}\label{exhaustivity} For every integer $k$, the submeasure $\nu_k$ is exhaustive. 
\end{prop}
\begin{proof} Fix  $k$ and suppose $(E_i)_i$ is a pairwise disjoint sequence of sets in $\mathcal B$. 
Fix some $s\geq k$ and $\epsilon >0$. Starting with $n_0=0$,
we use Lemma \ref{finding-thinness} repeatedly to construct an increasing sequence of integers $(n_l)_l$
and sets $B_l\in \mathcal B_{n_{l+1}}$ such that, for all $l$,  
$B_l$ is $(n_l,n_{l+1}, \nu_k)$-thin, and
\[
\limsup_{i\rightarrow \infty} \nu_{k}(E_i \setminus B_{l})\leq \frac{\epsilon}{2^{l+1}}.
\]
Let $I_l=\{ n_0,n_1, \ldots, n_l\}$. Since our norm $\| \cdot \|$ is unbounded, there is $l$ such that $\| I_l \|=M_s$.
Let $B= \bigcap_{i <l}B_i$. Then  the set $B$ is $(I_l, \nu_k)$-thin.  
By the $s$-thinness property of $\nu_k$ we have that $\nu_k(B)\leq 2^{-s}$.
Now, by the subadditivity of $\nu_k$ we have:

\[
\limsup_{i\rightarrow \infty}\nu_{k}(E_i\setminus B) \leq 
\sum_{j<l} \limsup_{i\rightarrow \infty}\nu_{k}(E_i \setminus B_j) \leq  (1-\frac{1}{2^{l+1}})\epsilon.
\]

Since $s\geq k$ and $\epsilon>0$ were arbitrary, it follows that $\limsup_{i\rightarrow \infty}\nu_k(E_i)=0$.
This completes the proof that $\nu_k$ is exhaustive. 
\end{proof}

Now, by combining  Proposition \ref{exhaustivity} and Corollary \ref{exhaustivity-rank-Schreier}
we obtain our main result. 

\begin{teo}\label{Glavna}
There are exhaustive submeasures on $\mathcal B$ of arbitrary high countable exhaustivity rank. 
\qed
\end{teo}

Let $\mathscr E$ denote the set of all exhaustive submeasure on $\mathcal B$. 
It is easy to see that $\mathscr E$ is a co-analytic subset of $[0,+\infty]^{\mathcal B}$ 
with the product topology. We now have the following  corollary.


\begin{coro}
The set $\mathscr E$ of exhaustive submeasures on $\mathcal B$ is not Borel.
\end{coro}

\begin{proof} 
Indeed, the function $\nu \mapsto {\rm rank}(\nu)$ is clearly a $\mathbf \Pi_1^1$-rank.
Since, by Theorem \ref{Glavna}, this function is unbounded below $\omega_1$,
by the rank method (see \cite{Kech:DST}, page 288) the set $\mathscr E$ is not Borel. 
\end{proof}

Suppose $\nu$ is strictly positive exhaustive submeasure on $\mathcal B$. 
In the standard way we define a metric $\rho$  on $\B$:   $\rho(E,F)=\nu(E\bigtriangleup F)$, for $E,F\in \B$.  
We use it to obtain a metric completion $\bar{\B}$ of $\B$. 
The continuous extension $\bar{\nu}$ of $\nu$ to $\bar{\B}$ is a strictly positive 
continuous submeasure of exhaustivity rank the same as $\nu$. Since, by \cite{Maharam} any
two continuous submeasures on a Maharam algebra $\mathcal M$ are absolutely continuous with respect
to each other, the exhaustivity rank is an algebraic invariant of $\mathcal M$. 
Therefore, from Theorem \ref{Glavna} we have the following corollary. 

\begin{coro}\label{Maharam-algebras}
 There are uncountably many pairwise non isomorphic separable atomless Maharam algebras.
 \qed
 \end{coro}
 
\section{Bounding the exhaustivity ranks}

As mentioned in the introduction Fremlin \cite{Fremlin-06204}
showed that the exhaustivity rank of Talagrand's submeasure from \cite{Talagrand} is at most $\omega^{\omega^2}$.
In this section we give bounds on the exhaustivity rank of our submeasures.  
If one wishes, one can then produce an explicit $\omega_1$-sequence of pairwise non isomorphic Maharam algebras. 
Thus, suppose $\nu$ is a submeasure on $\mathcal B$ satisfying the covering property 
and such that $\nu(T)\geq 8$. 
Suppose that $\| \cdot \|$ is an admissible norm, $N$ is an integer and $\nu$ 
satisfies the $N$-thinness property relative to $\| \cdot \|$. 
Recall that this means that there is an integer $M_N$ such that
that $\nu(X) \leq 2^{-N}$, for every set $X$ which is $(I,\nu)$-thin,
for some $I$ with $\| I \|=M_N$. 
Let 
$
\mathscr S=\{ F\in [\mathbb N]^{<\omega}: \| F \| < M_N\}
$
and  let $\beta=\rho(\mathscr S)$.  Recall that this means that $\beta$
is the least ordinal for which Player {\rm I} has a winning strategy in the game
$\mathcal G_{\beta}(\mathscr S)$ from Definition \ref{family-game}. 
For any $\epsilon >0$, we give an explicit bound
on the $(2^{-N}+\epsilon)$-exhaustivity rank of $\nu$. 

We start by making some definitions. Suppose $m$ is an integer and $A\in \mathcal A_m$. 
If $X\in \mathcal B$ we let $\nu(X | A)$ denote the relative submeasure of $X$ with respect to $A$,
i.e. $\nu (\pi_A^{-1}(X))$. Note that $\nu (X | A)=\nu (X\cap A | A)$. 
Suppose now $n >m$ and $\vec{C}=\{ C_r: m <r \leq n\}$ is a sequence
such that  $C_r\subseteq A$ and $C_r \in \mathcal B_r$, for all $m < r \leq n$. We let
\[
w(\vec{C}| A)=  \sum_{m < r \leq n} \nu ({C_r} | A).
\]

\begin{defi}\label{C-sequences} Suppose $m < n$ and $A\in \mathcal A_m$.
We let $\mathscr C_{m,n}(A)$ denote the collection of all sequences $\vec{C}=\{ C_r : m <r \leq n\}$
such that $C_r\subseteq A$, $C_r \in \mathcal B_r$, for all $m < r \leq n$, and $w(\vec{C}| A)\leq 4$.
We let $\mathscr C_m(A)=\bigcup \{ \mathscr C_{m,n}(A): m < n \}$.
\end{defi}

Suppose now $m < n \leq p$, $A\in \mathcal A_m$, $\vec{C}\in \mathscr C_{m,n}(A)$ and 
$\vec{D}\in \mathscr C_{m,p}(A)$.
We say that $\vec{D}$ is an {\em extension} of $\vec{C}$ if $\vec{D}\restriction (m,n]= \vec{C}$. 
If $\delta >0$ we say that $\vec{D}$ is a $\delta$-{\em proper extension} of $\vec{C}$ if $\vec{D}$ is
an extension of $\vec{C}$ and $w(\vec{D} | A) \geq w(\vec{C}  \vert A) + \delta$.

\begin{lemma}\label{relative-covering} Let $m$ be an integer and $A\in \mathcal A_m$.
Suppose $E\subseteq A$ and $\nu (E | A) <2$. Let $n>m$ be such that $E\in \mathcal B_n$.
Then there is $\vec{C}\in \mathscr C_{m,n}(A)$ such that $E \subseteq \bigcup \vec{C}$. 
\end{lemma}

\begin{proof} Let $E'= \pi_A^{-1}(E)$. Then $E'$ does not depend on coordinates $< m$ and $\nu (E') <2$.
By the $m$-covering property, we can fix an $m$-covering sequence $\{ C'_r : m < r \leq n\}$
of $E'$. Let $C_r=C_r'\cap A$, for all $m < r \leq n$.
Then $\vec{C}= \{ C_r : m < r \leq n\}$ is as required. 
\end{proof}

\begin{defi}\label{A-delta-game} Let $\alpha$ be an ordinal, $m$ an integer, $A\in \mathcal A_m$,
and $\delta >0$. The game $\mathcal G(\alpha,A,\delta)$  is played between two players
	${\rm I}$ and ${\rm II}$ as follows.
	\[
	\begin{array}[c]{ccccccccc}
	{\rm I} : & \alpha_0, C_0 \phantom{E_0} & \alpha_1,C_1 \phantom{E_1}& \ \ \cdots \ \ & \alpha_n,C_n \phantom{E_k}& \ \ \cdots \ \ \\
	\hline
	{\rm II} : & \phantom{\alpha_0, C_0} E_0 & \phantom{\alpha_1,C_1} E_1 & \ \ \cdots \ \ & \phantom{\alpha_n,C_n} E_n  & \ \ \cdots \ \ \\
	\end{array}
	\]
	
	\noindent Player ${\rm I}$ plays ordinals $\leq \alpha$ such that $\alpha_{n} \leq \alpha_{n-1}$, 
	and clopen sets $C_n\subseteq A$ such that $\nu (C_n | A)\leq 4$. 
	Player ${\rm II}$ plays clopen sets $E_n\subseteq A$. Player {\rm I} is required
	to play $\alpha_{n+1} < \alpha_n$ if	$\nu (\bigcup_{i <n} E_i  | A) < 2$
	and, either $\nu (\bigcup_{i \leq n}E_i | A)\geq 2$ or $\nu (E_n\setminus C_n |A)\geq \delta$. 
	In other case Player {\rm I} is allowed to play $\alpha_{n+1}=\alpha_n$. 
	Player {\rm I} wins if he can keep playing indefinitely by following these rules. 
\end{defi}

\begin{lemma}\label{A-delta-winning} Let $m$ be an integer, $A\in \mathcal A_m$ and $\delta>0$.
Let $k= \lceil 4/\delta\rceil$. Then Player {\rm I} has a winning strategy in $\mathcal G(\omega^{k+1},A,\delta)$.
\end{lemma}

\begin{proof} For $\vec{C}\in \mathscr C_m (A)$, let $k(\vec{C})$ be the least integer $l$ such that
$(l+1)\cdot \delta > 4- w(\vec{C} | A)$. 
To begin, Player {\rm I} plays $(\omega^{k+1},\emptyset)$. 
As long as $\nu(\bigcup_{i<n}E_i | A) <2$, Player {\rm I} plays ordinals $\alpha_n >0$. 
At stage $n>0$, if $\nu( \bigcup_{i <n} E_n | A) <2$, by Lemma \ref{relative-covering} 
there is $\vec{C}\in \mathscr C_m(A)$ such that $\bigcup_{i<n}E_n\subseteq \bigcup \vec{C}$. 
On the side Player {\rm I} keeps an integer $p(n)$ such that $E_i\in \mathcal B_{p(n)}$, 
for all $i<n$, and  a finite family ${\mathcal D}_n \subseteq \mathscr C_m(A)$
such that  every  $\vec{C}\in \mathscr C_{m,p(n)}(A)$ such that $\bigcup_{i <n}E_i \subseteq \bigcup \vec{C}$
extends a member of $\mathcal D_n$. 
Given these objects,  let us define $\alpha'_n$ to be the natural sum of
the ordinals $\omega^{k(\vec{C})}$, for $\vec{C}\in \mathcal D_n$, and let $\alpha_n=\alpha'_n+1$.
Player {\rm I} picks some $\vec{C}_n\in \mathcal D_n$,
sets $C_n=\bigcup \vec{C}_n$  and plays the pair $(\alpha_n, C_n)$. Suppose Player {\rm II} responds
by playing some $E_n$. If $\nu (\bigcup_{i \leq n}E_i| A) < 2$ and $\nu (E_n\setminus C_n |A) < \delta$,
Player {\rm I} simply repeats his previous move, i.e. he sets $(\alpha_{n+1},C_{n+1})=(\alpha_n,C_n)$.
He also sets $\mathcal D_{n+1}= \mathcal D_n$. 
If $\nu (\bigcup_{i \leq n} E_n| A)\geq 2$, Player {\rm I} sets $\alpha_{n+1}=0$ and $C_{n+1}=\emptyset$.
After that there are no requirements for him, so he keeps repeating this move indefinitely. 
Suppose now that $\nu (E_n\setminus C_n | A) \geq \delta$. Note that  any $\vec{D}\in \mathscr C_m(A)$
extending $\vec{C}_n$ and such that $E_n \subseteq \bigcup \vec{D}$ will be a $\delta$-proper extension
of $\vec{C}_n$, hence we'll have $k(\vec{D}) < k(\vec{C}_n)$. 
Let $p(n+1)$ be the least integer $p \geq p(n)$ such that $E_i\in \mathcal B_p$, for all $i \leq n$.
In order to define $\mathcal D_{n+1}$, Player {\rm I}
removes $\vec{C}_n$ from $\mathcal D_n$ and replaces it by all its $\delta$-proper extensions in 
$\mathscr C_{m,p(n+1)}(A)$. 
If $k(\vec{C}_n)=0$ there are no such extensions, so Player {\rm I} simply removes $\vec{C}_n$ from
$\mathcal D_n$. 
Also, observe that in the computation of $\alpha'_{n+1}$, we replaced $\omega^{k(\vec{C}_n)}$ by finitely
many ordinals of the form $\omega^l$, for $l< k(\vec{C}_n)$. It follows that $\alpha'_{n+1} < \alpha'_n$. 
Since $\alpha_{n+1}=\alpha'_{n+1}+1$, we also have that $\alpha_{n+1}< \alpha_n$
and, in addition, $\alpha_{n+1}\geq 1$.
Clearly, Player {\rm I} can play indefinitely by following this strategy. 
\end{proof}

\begin{defi}\label{m-delta-game} Let $\alpha$ be an ordinal, $m$ an integer, and $\delta >0$. 
The game $\mathcal H (\alpha,m,\delta)$  is played between two players
	${\rm I}$ and ${\rm II}$ as follows.
	\[
	\begin{array}[c]{ccccccccc}
	{\rm I} : & \alpha_0, B_0 \phantom{E_0} & \alpha_1,B_1 \phantom{E_1}& \ \ \cdots \ \ & \alpha_n,B_n \phantom{E_n}& \ \ \cdots \ \ \\
	\hline
	{\rm II} : & \phantom{\alpha_0, B_0} E_0 & \phantom{\alpha_1,B_1} E_1 & \ \ \cdots \ \ & \phantom{\alpha_n,B_n} E_n  & \ \ \cdots \ \ \\
	\end{array}
	\]
	
	\noindent Player ${\rm I}$ plays a strictly decreasing sequence of ordinals $< \alpha$ and sets $B_n\in \mathcal B$
	such that each $B_n$ is $(m,q_n)$-thin, for some $q_n >m$. At stage $n$, Player {\rm II} is required to 
	play some $E_n\in \mathcal B$ that is disjoint from the $E_i$, for $i <n$, and 
	 such that $\nu (E_n \setminus B_n)\geq \delta$. 
	The first player who cannot play following these  rules loses. 
\end{defi}

\begin{lemma}\label{m-delta-winning} Suppose $m$ is an integer and $\delta >0$. Let
 $k= \lceil 4\cdot | \mathcal A_m|/ \delta\rceil$. 
Then Player {\rm I} has a winning strategy in  $\mathcal H (\omega^{k+2}, m,\delta)$.
\end{lemma}

\begin{proof} Let $\delta'= \delta/| \mathcal A_m|$. By Lemma \ref{A-delta-winning}, we can 
fix a winning strategy $\sigma_A$ for Player {\rm I} in $\mathcal G(\omega^{k+1},A,\delta')$, for all $A\in \mathcal A_m$.
We describe a winning strategy $\sigma$ for Player {\rm I}
in $\mathcal H(\omega^{k+2},m,\delta)$.  We think of playing all the games 
$\mathcal G (\omega^{k+1},A,\delta')$  in parallel. 
In each of these games Player {\rm I} follows his winning strategy $\sigma_A$. 
If Player {\rm II} plays $E_n$ in $\mathcal H (\omega^{k+2},m,\delta)$ we consider
that he plays $E_n\cap A$ in the game $\mathcal G (\omega^{k+1},A,\delta')$. 
At stage $n$, let $(\alpha_n(A),C_n(A))$ be the $n$-th move of $\sigma_A$ in the game
$\mathcal G (\omega^{k+1},A,\delta')$.
For each $A\in \mathcal A_m$, let 
\[
H_n(A) = 
\begin{cases}
\bigcup_{i< n}E_i \cap A, & \mbox{ if } \nu(\bigcup_{i <n}E_i | A) \geq 2, \\
A \setminus C_n(A), &  \mbox { otherwise}.
\end{cases}
\]
Let $q_n$ be the least integer $q >m$ such that $H_n(A)\in \mathcal B_q$, for all $A\in \mathcal A_m$. 
Note that $\nu (H_n(A) | A) \geq 2$, for all $A\in \mathcal A_m$.
 Therefore, if we let $D_n(A)=A\setminus H_n(A)$, for all $A\in \mathcal A_m$, the set
\[
B_n= \bigcup \{ D_n(A): A\in \mathcal A_m\},
\] 
is $(m, q_n, \nu)$-thin. Let $\alpha_n$ be the natural sum of the $\alpha_n(A)$, for $A\in \mathcal A_m$.
The strategy $\sigma$ then plays $(\alpha_n,B_n)$.  
Suppose that Player {\rm II} responds by playing some $E_n$ disjoint from the $E_i$, for $i<n$,
and such that $\nu (E_n\setminus B_n)\geq \delta$. 
Then there must be some $A\in \mathcal A_m$ such that $\nu ((E_n \setminus B_n)\cap A)\geq \delta'$.
In particular, $\nu (E_n \cap H_n(A) | A)\geq \delta'$. 
If $H_n(A)= \bigcup_{i<n}E_i \cap A$, this is not possible since $E_n$ is disjoint from the $E_i$, for $i<n$.
Thus, it must be the case that $\nu(\bigcup_{i <n}E_i | A) <2$ and $\nu(E_n\setminus C_n(A)| A)\geq \delta'$.
This means that in the next move $\sigma_A$ must play some pair $(\alpha_{n+1}(A),C_{n+1}(A))$,
such that $\alpha_{n+1}(A)< \alpha_n(A)$. Since $\alpha_{n+1}(A')\leq \alpha_n (A')$, 
for all other $A'\in \mathcal A_m$, this means that $\alpha_{n+1} < \alpha_n$.
Therefore, by doing this, Player {\rm I} follows the rules in $\mathcal H (\omega^{k+2},m,\delta)$. 
Finally, let us note that, for all $A\in \mathcal A_m$,  the first move of $\sigma_A$ is $(\omega^{k+1},\emptyset)$.
Hence, the first move of  $\sigma$ is $\omega^{k+1}\cdot | \mathcal A_m | < \omega^{k+2}$.
Therefore $\sigma$ is a winning strategy for Player {\rm I} in $\mathcal H (\omega^{k+2},m,\delta)$,
as required. 

\end{proof}

We now introduced another game that will be used to bound the exhaustivity rank of our submeasure $\nu$.

\begin{defi}\label{exhaustivity-game} Let $\xi$ be an ordinal. 
The game $\mathcal E(\xi)$  is played between two players
	${\rm I}$ and ${\rm II}$ as follows.
	\[
	\begin{array}[c]{ccccccccc}
	{\rm I} : & \xi_0  \phantom{E_0} & \xi_1 \phantom{E_1}& \ \ \cdots \ \ & \xi_n \phantom{E_n}& \ \ \cdots \ \ \\
	\hline
	{\rm II} : & \phantom{\xi_0, } E_0 & \phantom{\xi_1} E_1 & \ \ \cdots \ \ & \phantom{\xi_n,} E_n  & \ \ \cdots \ \ \\
	\end{array}
	\]
	
	\noindent Player ${\rm I}$ plays a strictly decreasing sequence of ordinals $\leq \xi$ and 
	Player {\rm II} plays pairwise disjoint sets $E_n \in \mathcal B$ such that $\nu (E_n) \geq 2^{-N}+\epsilon$.  
	The first player who cannot play by following these rules loses. 
\end{defi}

Recall that we have assumed that $\| \cdot \|$ is an admissible norm and 
$\nu$  satisfies the $N$-thinness property relative to $\| \cdot \|$. 
We have defined 
$
\mathscr S=\{ F\in [\mathbb N]^{<\omega}: \| F \| < M_N\}
$
and let $\beta=\rho(\mathscr S)$.

\begin{lemma}\label{exhaustivity-winning} 
Player {\rm I} has a winning strategy in the game $\mathcal E({\omega^{\omega \cdot (\beta +1)}})$.
\end{lemma}

\begin{proof} Let $\tau$ be a winning strategy for Player {\rm I} in the game $\mathcal G_\beta (\mathscr S)$ 
from Definition \ref{family-game}.
For every $m$ and $\delta>0$, fix a winning strategy $\sigma_{m,\delta}$ for Player {\rm I} in
$\mathcal H (\omega^\omega, m,\delta)$.
We combine those strategies into a winning strategy for Player {\rm I} in $\mathcal E({\omega^{\omega \cdot (\beta +1)}})$. 
Let us write $\epsilon_i$ for $\epsilon /2^{i+1}$. 
To avoid excessive notation, let us introduce some dynamic variables.
First, $l$ will denote an integer, $F$ a set of integers of size $l+1$, and $\{ m_0,\ldots,m_l\}$
will denote the increasing enumeration of $F$. Also, 
$\vec{\gamma}$ will denote a decreasing sequence $(\gamma_0,\ldots, \gamma_{l})$
of ordinals $\leq \beta$ of length $l+1$. 
We will have that $(\gamma_0,m_0,\ldots, \gamma_{l-1},m _{l-1},\gamma_l)$ is a position
in $\mathcal G_{\beta}(\mathscr S)$ in which Player {\rm I} uses his strategy $\tau$. 
In particular, we will have that $\{ m_0,\ldots, m_{l-1}\}\in \mathscr S$, but $F$ itself may not be in $\mathscr S$.
For each $i<l$ we will also fix a variable $\pi_i$ denoting a certain 
position in the game $\mathcal H (\omega^\omega, m_i,\epsilon_i)$,
in which Player {\rm I} uses his winning strategy $\sigma_{m_i,\epsilon_i}$ and
Player {\rm II} plays some of the $E_j$ from the game $\mathcal E(\omega^{\omega \cdot (\beta+1)})$.
We denote the last move of Player {\rm I} in $\pi_i$ by  $(\alpha_i, B_i)$. 
We will also have that $B_i \in \mathcal B_{m_{i+1}}$ and is $(m_i,m_{i+1},\nu)$-thin.
Given the value of all these variables at stage $n$ we will compute a certain ordinal $\xi_n$
which will be the move of Player {\rm I} at that stage. 
Depending on the next move of Player {\rm II} we will reset these variables for the next
stage of the game. 

To begin, set $l=1$, $\gamma_0=\beta$, $m_0=0$. Let $\gamma_1$ be the response of $\tau$
if Player {\rm II} plays $m_0$ as his first move in the game $\mathcal G_{\beta}(\mathscr S)$.  
Set $\vec{\gamma}$ to be $(\gamma_0,\gamma_1)$.
Set $\pi_0$ to be the position in $\mathcal H (\omega^\omega, 0,\epsilon_0)$
after the first move of Player {\rm I} given by the strategy $\sigma_{0,\epsilon_0}$. 
Set  $m_1$ to be the least integer $q$ such that $B_0\in \mathcal B_q$. Set $F$ to be $\{ m_0,m_1\}$. 

Now, suppose we are at some stage $n$ of the game $\mathcal E(\omega^{\omega \cdot (\beta +1)})$.
Given the current values of the above variables, let  $s$ be such that the first move of 
$\sigma_{m_l,\epsilon_l}$ is $<\omega^s$. As his $n$-th move in $\mathcal E(\omega^{\omega \cdot (\beta +1)})$
Player {\rm I}  plays $\xi_n$ equal to: 
\begin{equation}
\omega^{\omega \cdot \gamma_0}\cdot \alpha_0 \oplus \omega^{\omega \cdot \gamma_1}\cdot \alpha_1 
\oplus \ldots 
\oplus \omega^{\omega \cdot \gamma_{l-1}}\cdot \alpha_{l-1}
\oplus \omega^{\omega \cdot \gamma_l +s}.
\label{move}
\end{equation}

Now, suppose Player {\rm II} responds by playing some $E_n$ disjoint from the $E_i$, for $i<n$,
and such that $\nu (E_n)\geq 2^{-N} +\epsilon$. Let us describe how the above variables are reset. 
Consider the current values of the $B_i$, for $i<l$. 
\medskip

\noindent{\bf Case 1.} Suppose first that $\nu (E_n\setminus B_i) < \epsilon_i$, for all $i$.
Note that the set $B=\bigcap \{ B_i :i <l\}$ is $(F,\nu)$-thin. If $F\notin \mathscr S$ we have that 
$\| F \| = M_N$ and, hence, by the $N$-thinness property of $\nu$, we conclude that $\nu(B)\leq 2^{-N}$. 
But then we would have:
\[
\nu (E_n)\leq \nu (B) + \nu (E_n \setminus B) \leq 2^{-N} + \sum_{i <l}\epsilon_i < 2^{-N} +\epsilon,
\]
which is a contradiction. 
Now, if $F\in \mathscr S$ then $m_l$ is a legitimate move for Player {\rm II}
in the position $(\gamma_0,m_0,\ldots, m_{l-1},\gamma_l)$ of $\mathcal G_{\beta}(\mathscr S)$. 
We now reset the new value of $l$ to be $l+1$. We set $\gamma_{l+1}$ to be the move
of $\tau$ in the position $(\gamma_0,m_0,\ldots,\gamma_l,m_l)$ of the game $\mathcal G_{\beta}(\mathscr S)$. 
We start a  run $\pi_l$ of $\mathcal H(\omega^\omega, m_l, \epsilon_l)$ by letting the strategy
$\sigma_{m_l,\epsilon_l}$ make the first move, say  $(\alpha_l,B_l)$, in that game. 
We let $m_{l+1}$ be the least integer $q \geq m_l$
such that $B_l\in \mathcal B_q$. We then add $m_{l+1}$ to $F$.
All other variables are kept unchanged. Let us consider the effect of these changes
on \eqref{move}. The first $l$ terms have not changed. We have replaced $\omega^{\omega \cdot \gamma_l +s}$
by 
\[
\omega^{\omega \cdot \gamma_l}\cdot \alpha_l \oplus \omega^{\omega \cdot \gamma_{l+1}+ s'}
\]
for some integer $s'$. Note that $\alpha_l < \omega^s$ and $\gamma_{l+1} < \gamma_l$,
hence the value of \eqref{move} decreases in the next stage of the game, i.e. $\xi_{n+1} <\xi_n$.
\medskip

\noindent {\bf Case 2.} Suppose now that $\nu (E_n\setminus B_i)\geq \epsilon_i$, for some $i$. 
Let $j$ be the least such $i$. This means that $E_n$ is a legitimate move for Player {\rm II}
in the current position $\pi_j$ of $\mathcal H(\omega^\omega,m_j,\epsilon_j)$. We then let Player {\rm II} play $E_n$ in this position and we let 
$\sigma_{m_j,\epsilon_j}$ respond to this move. We set the resulting position to be our new $\pi_j$. 
We set the new value of $l$ to be $j+1$. 
We keep all the positions $\pi_i$, for $i <j$, unchanged and we erase the positions $\pi_i$, for $i >j$.
We keep the values of the $\gamma_i$, for $i \leq j+1$, unchanged and we erase the  $\gamma_i$, for $i >j+1$. 
We keep all the $m_i$, for $i\leq j$, unchanged. For our new $m_{j+1}$ we pick the least integer
$q$ such that the new $B_j$ belongs to $\mathcal B_q$. We erase all the $m_i$, for $i > j+1$.
Finally, we set $F=\{ m_0,\ldots, m_{j+1}\}$. In order to estimate the effect of these changes
to \eqref{move} let us denote by $\alpha_l'$ the old value of $\alpha_l$ and by $\alpha''_l$ the new value
of $\alpha_l$. Let us also denote by $\alpha'_{l+1}$ the old value of $\alpha_{l+1}$. 
The first $l-1$ terms of \eqref{move} have not changed. 
In the $l$-th term we replaced $\omega^{\omega \cdot \gamma_l}\cdot \alpha'_l$ by 
$\omega^{\omega \cdot \gamma_l}\cdot \alpha''_l$ and in the $l+1$-th term we replaced
$\omega^{\omega \cdot \gamma_{l+1}}\cdot \alpha'_{l+1}$ by 
$\omega^{\omega \cdot \gamma_{l+1} +s}$, for some integer $s$. We erased
all later terms. Now, note that $\alpha''_l < \alpha'_l$ and $\gamma_{l+1}< \gamma_l$,
hence,
\[
\omega^{\omega \cdot \gamma_l}\cdot \alpha''_l \oplus \omega^{\omega \cdot \gamma_{l+1} +s} < 
\omega^{\omega \cdot \gamma_l}\cdot \alpha''_l + \omega^{\omega \cdot \gamma_l} 
=\omega^{\omega \cdot \gamma_l}\cdot (\alpha''_l+1) \leq
\omega^{\omega \cdot \gamma_l}\cdot \alpha'_l.
\]
This means that the value of \eqref{move} decreases in the next stage of the game, i.e. $\xi_{n+1} <\xi_n$.
Thus, Player {\rm I} can continue playing in this way as long as Player {\rm II} plays pairwise disjoint
sets $E_n$ with $\nu (E_n) \geq 2^{-N}+\epsilon$. This completes the proof of Lemma \ref{exhaustivity-winning}.
 
\end{proof}

Now, combining Corollary \ref{rank-sums}, Lemma \ref{Schreier-ranks}, Corollary \ref{exhaustivity-rank-Schreier}
and  Lemma \ref{exhaustivity-winning}, we obtain the following.

\begin{coro}\label{final-bounds} Suppose $0<\alpha <\omega_1$. Let $\| \cdot \|_\alpha$
be the admissible norm derived from the $\alpha$-th Schreier family and let $\nu^\alpha$
be the associated exhaustive submeasure. Then 
\[
\omega^\alpha \leq {\rm rk}(\nu^\alpha) \leq \omega^{\omega \cdot (\alpha +1)^\omega}.
\]
\qed
\end{coro}

 \bibliographystyle{abbrv}

\bibliography{mc1}

\end {document}